\documentclass[11pt,leqno]{article}
\usepackage{amsmath, amscd, amsthm, amssymb, graphics, xypic, mathrsfs, setspace, fancyhdr, times}

\DeclareFontFamily{OT1}{pzc}{}
\DeclareFontShape{OT1}{pzc}{m}{it}%
             {<-> s * [1.195] pzcmi7t}{}
\DeclareMathAlphabet{\mathscr}{OT1}{pzc}%
                                 {m}{it}
\usepackage[pagebackref=true]{hyperref}
\hypersetup{backref}

\newcommand{\tensor}{\otimes}
\newcommand{\colim}{\operatorname{colim}}
\newcommand{\Spec}{\operatorname{Spec}}
\newcommand{\isomto}{{\stackrel{\sim}{\;\longrightarrow\;}}}
\newcommand{\isomt}{{\stackrel{{\scriptscriptstyle{\sim}}}{\;\rightarrow\;}}}

\renewcommand{\O}{{\mathcal O}}
\renewcommand{\L}{{\mathcal L}}
\renewcommand{\hom}{\operatorname{Hom}}

\newcommand{\Z}{{\mathbb Z}}
\newcommand{\aone}{{\mathbb A}^1}
\newcommand{\pone}{{\mathbb P}^1}

\renewcommand{\1}{{\mathbf{1}}}

\newcommand{\gm}{{{\mathbb G}_{m}}}
\newcommand{\So}{{\mathbb S}^0}

\newcommand{\ho}[1]{\mathscr{H}({#1})}
\newcommand{\hop}[1]{\mathscr{H}_{\bullet}({#1})}
\newcommand{\DM}{{\mathbf{DM}}}
\newcommand{\M}{{\mathbf{M}}}
\newcommand{\Sp}{{\mathscr{Sp}}}
\newcommand{\bpi}{\boldsymbol{\pi}}
\newcommand{\Daone}{{\mathbf D}_{\aone}}
\newcommand{\eff}{\operatorname{eff}}
\newcommand{\Nis}{\operatorname{Nis}}
\newcommand{\SH}{{\mathbf{SH}}}
\newcommand{\CH}{{\widetilde{CH}}}
\newcommand{\Ch}{{\mathscr{Ch}}}
\newcommand{\Fun}{\mathscr{Fun}}
\renewcommand{\deg}{\operatorname{deg}}
\newcommand{\tdeg}{\widetilde{\deg}}

\newcommand{\Sm}{\mathscr{Sm}}
\newcommand{\Cor}{\mathscr{Cor}}
\newcommand{\Spc}{\mathscr{Spc}}

\newcommand{\Ab}{\mathscr{Ab}}
\newcommand{\Set}{\mathscr{Set}}
\newcommand{\K}{{{\mathbf K}}}
\newcommand{\I}{{\mathbf I}}
\renewcommand{\H}{{{\mathbf H}}}

\newcommand{\hsnis}{\mathscr{H}_s^{\Nis}(k)}
\newcommand{\hspnis}{\mathscr{H}_{s,\bullet}^{\Nis}(k)}

\newcommand{\F}{{\mathcal F}}

\newcommand{\Zn}{\Z \langle n \rangle}

\newcounter{intro}
\theoremstyle{plain}
\newtheorem{thm}{Theorem}[subsection]

\newtheorem{lem}[thm]{Lemma}
\newtheorem{cor}[thm]{Corollary}
\newtheorem{prop}[thm]{Proposition}

\newtheorem*{thm*}{Theorem}
\newtheorem*{problem*}{Problem}
\newtheorem*{question*}{Question}

\newtheorem{thmintro}{Theorem}

\theoremstyle{definition}
\newtheorem{defn}[thm]{Definition}

\newtheorem{notation}[thm]{Notation}

\theoremstyle{remark}
\newtheorem{rem}[thm]{Remark}

\newtheorem{ex}[thm]{Example}
\newtheorem{entry}[thm]{}

\numberwithin{equation}{section}

\begin{document}
\pagestyle{fancy}
\renewcommand{\sectionmark}[1]{\markright{\thesection\ #1}}
\fancyhead{}
\fancyhead[LO,R]{\bfseries\footnotesize\thepage}
\fancyhead[LE]{\bfseries\footnotesize\rightmark}
\fancyhead[RO]{\bfseries\footnotesize\rightmark}
\chead[]{}
\cfoot[]{}
\setlength{\headheight}{1cm}

\author{\begin{small}Aravind Asok\thanks{Aravind Asok was partially supported by National Science Foundation Awards DMS-0900813 and DMS-0966589.}\end{small} \\ \begin{footnotesize}Department of Mathematics\end{footnotesize} \\ \begin{footnotesize}University of Southern California\end{footnotesize} \\ \begin{footnotesize}Los Angeles, CA 90089-2532 \end{footnotesize} \\ \begin{footnotesize}\url{asok@usc.edu}\end{footnotesize} \and \begin{small}Christian Haesemeyer\thanks{Christian Haesemeyer was partially supported by National Science Foundation Award DMS-0966821.}\end{small} \\ \begin{footnotesize}Department of Mathematics\end{footnotesize} \\ \begin{footnotesize}University of California, Los Angeles\end{footnotesize} \\ \begin{footnotesize}Los Angeles, CA 90095-1555 \end{footnotesize} \\ \begin{footnotesize}\url{chh@math.ucla.edu}\end{footnotesize}}

\title{{\bf The $0$-th stable $\aone$-homotopy sheaf \\ and quadratic zero cycles}}
\date{}
\maketitle

\begin{abstract}
We study the $0$-th stable $\aone$-homotopy sheaf of a smooth proper variety over a field $k$ assumed to be infinite, perfect and to have characteristic unequal to $2$.  We provide an explicit description of this sheaf in terms of the theory of (twisted) Chow-Witt groups as defined by Barge-Morel and developed by Fasel.  We study the notion of rational point up to stable $\aone$-homotopy, defined in terms of the stable $\aone$-homotopy sheaf of groups mentioned above.  We show that, for a smooth proper $k$-variety $X$, existence of a rational point up to stable $\aone$-homotopy is equivalent to existence of a $0$-cycle of degree $1$.
\end{abstract}

\begin{footnotesize}
\tableofcontents
\end{footnotesize}

\section{Introduction}
Assume $k$ is a field and $X$ is a smooth proper variety over $k$.  One says that $X$ has a $0$-cycle of degree $1$ if there exist finitely many finite extensions $L_i$ over $k$ of coprime degrees such that $X(L_i)$ is non-empty for each $i$.  Existence of a $0$-cycle of degree $1$ is inherently a motivic homological condition.  Indeed, by definition, $X$ has a $0$-cycle of degree $1$ if and only if the degree map $\operatorname{deg}: CH_0(X) \to \Z$ is split surjective, and a $0$-cycle of degree $1$ is a choice of splitting, or equivalently a lift of $1 \in \Z$.  Friedlander-Voevodsky duality implies that the group $CH_0(X)$ is a motivic homology group (we make this much more precise in \S \ref{ss:suslinhomology} using Voevodsky's derived category of motives \cite{MVW}).

If $X(k)$ is non-empty, then the degree map $\operatorname{deg}: CH_0(X) \to \Z$ is split surjective: sending $1 \in \Z$ to $x$ determines a splitting.  However, the converse is not true in general (see, e.g., \cite[\S 5]{CTC} for a counterexample).  Thus, if $X$ has no $0$-cycle of degree $1$, $X$ cannot have a $k$-rational point; loosely speaking, we will say that existence of a $0$-cycle of degree $1$ is a motivic homological obstruction to existence of a $k$-rational point.

Just as integral singular homology is the target of the Hurewicz homomorphism from (stable) homotopy groups, motivic homology is the target of a Hurewicz-style map from motivic stable homotopy groups \cite{VICM, MIntro}.  Furthermore, the degree homomorphism $CH_0(X) \to \Z$ is the pushforward map in motivic homology for the map $X \to \Spec k$.  Being covariantly functorial, motivic stable homotopy groups are also furnished with an analog of the degree map.

\begin{question*}
What kind of obstruction to existence of a rational point is provided by motivic stable homotopy groups of a smooth proper variety?
\end{question*}

To answer the question, we study the notion of ``rational point up to stable $\aone$-homotopy."  To be more precise, let $\pi_0^s(\Sigma^{\infty}_{\pone}X_+)$ be the zeroth motivic stable homotopy group of $X$ with a disjoint base-point attached (see \S \ref{ss:formalism}).  The structure morphism $X \to \Spec k$ induces a pushforward map $\pi_0^s(\Sigma^{\infty}_{\pone}X_+) \to \pi_0^s(\Sigma^{\infty}_{\pone}\Spec k_+)$.  A rational point up to stable $\aone$-homotopy is a choice of splitting of this map.  In Section \ref{ss:rationalpointsuptostablehomotopy} we use a slightly different though equivalent definition; there is even a more-or-less elementary definition along the lines of the first definition of $0$-cycle of degree $1$ we gave above.  To our knowledge, the notion of rational point up to stable $\aone$-homotopy was first explicitly considered by R\"ondigs; see \cite[Theorem 5.1]{Roendigs} and the discussion immediately preceding his theorem statement.

\begin{thmintro}[See Theorem \ref{thm:stablerationalpointnonformallyrealcase}]
\label{thmintro:rationalpointsuptostablehomotopy}
If $X$ is a smooth proper variety over a field $k$ assumed to be infinite, perfect and to have characteristic unequal to $2$, then $X$ has a rational point up to stable $\aone$-homotopy if and only if $X$ has a $0$-cycle of degree $1$.
\end{thmintro}

The $0$-th stable $\aone$-homotopy group of $X$ as discussed above is the set of sections over $k$ of a (Nisnevich) sheaf of abelian groups.  Theorem \ref{thmintro:rationalpointsuptostablehomotopy} is a consequence of a concrete description of this sheaf, which will be written $\bpi_0^s(\Sigma^{\infty}_{\pone}X_+)$ in the sequel, and detailed study of the functoriality.

To motivate our description, we go back to the Hurewicz-style homomorphism from the $0$-th stable $\aone$-homotopy group of a smooth variety $X$ to the zeroth motivic homology group mentioned above.  The $0$-th Suslin homology sheaf of a smooth proper $k$-scheme $X$, denoted $\H_0^S(X)$, is a sheafification of the Chow group of $0$-cycles on $X$ (see \S \ref{ss:abelianization} and Lemma \ref{lem:suslinchow}).  The Hurewicz functor gives rise to a sheafified Hurewicz homomorphism from the $0$-th stable $\aone$-homotopy sheaf of a smooth proper variety $X$ to its zeroth Suslin homology sheaf.  On the other hand, a fundamental computation of Morel describes the $0$-th stable $\aone$-homotopy sheaf of a point (more precisely, the motivic sphere spectrum) in terms of so-called Milnor-Witt K-theory sheaves, which combine Milnor K-theory and powers of the fundamental ideal in the Witt ring \cite{MICM}.

{\em A priori}, it seems reasonable to search for a description of $\bpi_0^s(\Sigma^{\infty}_{\pone}X_+)$ involving the features of $0$-cycles together with additional ``quadratic" data.  Our description of the $0$-th stable $\aone$-homotopy sheaf of a smooth proper variety is exactly in these terms and uses what are called (twisted) Chow-Witt groups.  As the name suggests, Chow-Witt groups combine aspects of the theory of algebraic cycles (Chow groups) with aspects of the theory of quadratic forms (Witt groups), and the twist refers to a choice of line bundle on $X$.  The main new complication is that the resulting theory is unoriented in the sense of cohomology theories, though confusingly Chow-Witt groups have also been called oriented Chow groups.

For a smooth proper $k$-scheme $X$, one defines the {\em degree $0$ Chow-Witt group} $\CH_0(X)$.  Very roughly speaking, this group is a quotient of the free abelian group on pairs $(x,q)$ where $x$ is a closed point of $X$ and $q$ is a quadratic form over the residue field $\kappa_x$ (for the precise definition, see Definition \ref{defn:quadraticzerocycles}); this construction encapsulates the sense in which the term ``quadratic" is used in the title of the paper.  If $X = \Spec k$, then $\CH_0(\Spec k)$ coincides with the Grothendieck-Witt group of symmetric bilinear forms over $k$ (in agreement with Morel's computation).  The resulting theory of Chow-Witt groups admits forgetful maps to Chow groups and has reasonable functoriality properties (e.g., pushforwards for proper maps in the appropriate situations).  Our main computational result is summarized in the following theorem.

\begin{thmintro}[See Theorem \ref{thm:main}]
\label{thmintro:main}
Suppose $k$ is an infinite perfect field having characteristic unequal to $2$, and $X$ is a smooth proper $k$-variety. For any separable, finitely generated extension $L/k$, there are isomorphisms
\[
\bpi_0^s(\Sigma^{\infty}_{\pone}X_+)(L) \isomto \CH_0(X_L)
\]
functorial with respect to field extensions.
\end{thmintro}

Theorem \ref{thmintro:main} is proven in three essentially independent steps, which correspond to the three sections of the paper subsequent to this introduction.  Using Morel's stable $\aone$-connectivity theorem and duality statements, much of the work can be viewed as a reduction to Morel's computations of some stable $\aone$-homotopy sheaves of the sphere spectrum.

First, we reduce to a corresponding homological statement: we show that in-so-far as our computation is concerned, we may replace the stable $\aone$-homotopy category by a corresponding stable $\aone$-derived category.  The latter is a variant of Voevodsky's triangulated category of motivic complexes (with the Tate motive inverted).  The stable $\aone$-derived category was originally conceived by Morel (see \cite[\S 5.2]{Morelpi0}), and a detailed construction appears in the work of Cisinski-D\'eglise \cite{CisinskiDeglise1}.  We refer to the passage from the stable $\aone$-homotopy category to the stable $\aone$-derived category as abelianization since it can be obtained by taking the derived functors of the free abelian group functor.  The main result of Section \ref{s:preliminaries} is Theorem \ref{thm:ponestablehurewiczisomorphism}, which makes precise the statements concerning Hurewicz homomorphisms mentioned above. With the exception of Sections \S \ref{s:preliminaries} and the end of \S \ref{s:twistedthomisomorphisms}, the reader so inclined can entirely avoid the stable $\aone$-homotopy category.

Second, via Spanier-Whitehead duality or Poincar\'e duality, homological computations can be turned into cohomological computations.  For smooth projective varieties, there is an analog of Atiyah's classical explicit description of the Spanier-Whitehead dual in terms of Thom spaces of the negative tangent bundle; the precise results we need are described in \S \ref{ss:atiyahduality}.  Theorem \ref{thm:aonehomologyintermsofthomspaces} gives an explicit description of the sections of the $0$-th stable $\aone$-homology sheaf of a smooth projective scheme in terms of the Thom space of a stable normal bundle.

Third, the cohomology of the Thom spaces in question fits into a general theory of twisted Thom isomorphisms that we develop here---in the unoriented setting the Thom isomorphisms we construct involve cohomology with coefficients twisted by a local system just like Poincar\'e duality on unoriented manifolds.  However, rather than defining the notion of a local system in $\aone$-homotopy theory, which would require some additional effort, we take a shortcut that involves delving into the formalism of Chow-Witt groups twisted by a line bundle.  Without the twist (or, rather, with the twist by the canonical line bundle) this theory was defined by Barge-Morel \cite{BargeMorel}.  A more general theory that keeps track of the line bundle twist was developed by Fasel \cite{Fasel1,Fasel2}, and involves the Balmer-Witt groups of a triangulated category with duality.  Beginning with a review of the analog of our computation in classical algebraic topology, the theory and results just mentioned are developed in Section \ref{s:twistedthomisomorphisms}.

The functoriality statements regarding the zeroth stable $\aone$-homotopy sheaf are summarized in the following result.

\begin{thmintro}[See Theorem \ref{thm:compatibilities}]
\label{thmintro:compatibilities}
The isomorphisms of \textup{Theorem \ref{thmintro:main}} satisfy the following compatibilities.
\begin{itemize}
\item[A)] The ``Hurewicz" homomorphism $\bpi_0^s(\Sigma^{\infty}_{\pone}X_+) \to \H_0^S(X)$ induces the forgetful map $\CH_0(X_L) \to CH_0(X_L)$ upon evaluation on sections over a separable finitely generated extension $L/k$. This homomorphism is surjective for any such field extension $L/k$.
\item[B)] The pushforward map $\bpi_0^s(\Sigma^{\infty}_{\pone}X_+) \to \bpi_0^s(\Sigma^{\infty}_{\pone}\Spec k_+)$ coincides with the pushforward map $\tdeg: \CH_0(X_L) \to \CH_0(\Spec L)$ via Morel's identification of $\bpi_0^s(\Sigma^{\infty}_{\pone}\Spec k_+)(L)$ with the $GW(L)$.
\item[C)] The two identifications just mentioned are compatible, i.e., the diagram
\[
\xymatrix{
\CH_0(X_L) \ar[r]\ar[d]^-{\tdeg} & CH_0(X_L) \ar[d]^-{\deg} \\
GW(L) \ar[r]^-{\operatorname{rk}} & \Z
}
\]
commutes.
\end{itemize}
\end{thmintro}

The proof Theorem \ref{thmintro:compatibilities} follows from Theorem \ref{thmintro:main} together with good choices of definitions of the objects under consideration.  Indeed, Theorem \ref{thmintro:compatibilities}A follows from the main computation and the definition of the Hurewicz homomorphism together with the discussion of abelianization in \S \ref{ss:abelianization}.  Theorem \ref{thmintro:compatibilities}B is a consequence of the duality formalism and \ref{thmintro:compatibilities}C follows by combining parts B and A.  We develop the necessary preliminaries in \S \ref{ss:formalism}.  Finally Theorem \ref{thmintro:rationalpointsuptostablehomotopy} follows from Theorem \ref{thmintro:main} by careful analysis of the Gersten resolution of the zeroth stable $\aone$-homotopy sheaf.

\subsubsection*{Relationship with other work}
In \cite{AH}, we showed that existence of a $k$-rational point was detected by the stable $\aone$-homotopy category of $S^1$-spectra or even the rationalized variant of this category.  Combining Theorem \ref{thmintro:rationalpointsuptostablehomotopy} with the main result from \cite{AH}, we see that the difference between rational points and $0$-cycles of degree $1$ arises from the passage from $S^1$-spectra to $\pone$-spectra.  Since the latter category ``has transfers" in a sense we describe here, this result corroborates a principle suggested in \cite{LevineSlices}.  There are a number of tools available to analyze this transition and one can hope to construct an obstruction theory for lifting $0$-cycles of degree $1$ to $k$-rational points.

\subsubsection*{Acknowledgements}
This paper originally began as a joint project between the authors and Fabien Morel; a preliminary version of the work herein was presented as such by the first named author in a talk at the Oberwolfach workshop ``Motives and the homotopy theory of schemes" in May 2010 ({\em cf.} \cite{OberwolfachReport}).  We thank Morel for his collaboration in the early stages of the project.  The first author would also like to thank the University of Essen, and especially Marc Levine, for providing an excellent working environment while portions of this work were completed.  We would also like to thank Richard Elman and Sasha Merkurjev for answering our questions regarding the theory of quadratic forms, and Jean Fasel for answering our questions regarding Chow-Witt theory.

\section{Preliminaries and reductions}
\label{s:preliminaries}
In this section, we recall (/develop) the formalism necessary in the rest of the paper.  In Section \ref{ss:formalism}, we recall the definitions and necessary properties of all the different categories of motivic origin that are required in our study.  This section will also serve to fix the notation for the rest of the paper.  In Section \ref{ss:suslinhomology}, we recall the proof of the analog of our main theorem in the context of Suslin homology; the main results are Lemmas \ref{lem:suslinchow} and \ref{lem:suslindegree}.  The proofs of these results will be used as templates for the proofs of our more general results in stable $\aone$-homotopy theory.  Finally Section \ref{ss:abelianization} studies the Hurewicz homomorphism (introduced in Section \ref{ss:formalism}) from stable $\aone$-homotopy sheaves of groups to stable $\aone$-homology sheaves.  The main result, i.e., Theorem \ref{thm:ponestablehurewiczisomorphism} allows us to focus our attentions on the stable $\aone$-derived category.

\subsection{Motivic categories}
\label{ss:formalism}
In this paper, we use a number of different categories of motivic or $\aone$-homotopy theoretic origin.  For the reader's benefit, Diagram \ref{eqn:motiviccategories} displays the categories under consideration and their mutual relationships.

\begin{equation}
\label{eqn:motiviccategories}
\xymatrix{
                    &                                  & \hsnis \ar[d]\ar[r]    & \SH_{S^1}(k) \ar[r]\ar[d]  & \mathbf{D}(\Ab(k)) \ar[r]\ar[d] & \mathbf{D}(\Cor(k)) \ar[d]\\
\Sm_k \ar[r]^-{} & \Spc_k \ar[r]^-{}\ar[ur]^-{} & \ho{k} \ar[r]^-{} & \SH_{s}(k)\ar@<.4ex>[r]^-{}\ar[d]^-{} & \Daone^{\eff}(k) \ar@<.4ex>[l]^-{}\ar@<.4ex>[r]^-{}\ar[d]^-{} & \DM^{\eff}(k) \ar@<.4ex>[l]^-{}\ar[d]^-{}\\
& & & \SH(k) \ar[r]^-{} & \Daone(k) \ar[r]^-{} & \DM(k)
}
\end{equation}

In the diagram, the nine categories below and to the right of $\SH_{S^1}(k)$ are homotopy categories of stable monoidal model categories and the functors indicated by arrows pointing down or to the right are monoidal in an appropriate sense.  The rightward pointing arrows in this group (and their composites) are the Hurewicz functors referred to in the introduction.  The categories in the third column from the left are homotopy categories of model categories that are not stable, and the categories in the first and second column are the geometric categories from which all of the homotopy categories are built.  The passage from the homotopy categories in the first row to the homotopy categories in the second and third rows is achieved by means of a Bousfield localization.

\subsubsection*{The unstable categories}
\begin{entry}[Smooth schemes and spaces]
We write $\Sm_k$ for the category of schemes separated, smooth and of finite type over $k$.  The category $\Spc_k$ is the category of simplicial Nisnevich sheaves of sets on $\Sm_k$; we refer to objects of $\Spc_k$ as {\em $k$-spaces} or simply {\em spaces} if $k$ is clear from context.  The functor $\Sm_k \to \Spc_k$ of Diagram \ref{eqn:motiviccategories} is defined as follows: send a smooth scheme $X$ to the corresponding representable functor $\hom_{\Sm_k}(\cdot,X)$ and then take the associated constant simplicial object (all face and degeneracy maps are the identity).  The Yoneda lemma shows that the functor so defined is fully faithful.  We systematically abuse notation by identifying the category of smooth schemes with its essential image in $\Spc_k$.

We also introduce the category of {\em pointed $k$-spaces}, denoted $\Spc_{k,\bullet}$; objects of this category are pairs $(\mathscr{X},x)$ consisting of a $k$-space $\mathscr{X}$ together with a morphism of spaces $x: \Spec k \to \mathscr{X}$.  The forgetful functor $\Spc_{k,\bullet} \to \Spc_k$ has a left adjoint sending a space $\mathscr{X}$ to $\mathscr{X}_+$, which is $\mathscr{X}$ with a disjoint base-point attached.

If $\mathscr{X}$ is a space and $F^\bullet$ is a complex of Nisnevich sheaves of abelian groups on $\Sm_k$, then we define hypercohomology groups $\mathbb{H}^p(\mathscr{X},F^\bullet)$ as (hyper-)extension groups in the category of Nisnevich sheaves of abelian groups on $\Sm_k$:
\[
\mathbb{H}^p(\mathscr{X},F^\bullet) = \mathrm{Ext}^p_{\Nis}(\Z(\mathscr{X}), F^\bullet).
\]
If $X$ is a scheme, this hypercohomology is canonically isomorphic to that obtained by the usual definition.
\end{entry}

\begin{entry}
The category $\hsnis$ is the (unpointed) homotopy category of Nisnevich simplicial sheaves as constructed by Joyal-Jardine (see, e.g., \cite[\S 2 Theorem 1.4]{MV}).  One equips $\Spc_k$ with the injective local model structure (cofibrations are monomorphisms, weak equivalences are stalkwise weak equivalences of simplicial sets, and fibrations are determined by the right lifting property) and realizes $\hsnis$ as the associated homotopy category.  Given two spaces $\mathscr{X}$ and $\mathscr{Y}$, we write $[\mathscr{X},\mathscr{Y}]_s$ for the set of morphisms between the resulting objects in $\hsnis$.  We write $\hspnis$ for the pointed variant of $\hsnis$, and if it is not already clear from context, to distinguish morphisms in this category we will explicitly specify the base-point.
\end{entry}

\begin{entry}[Unstable $\aone$-homotopy category]
The category $\ho{k}$ is the (unpointed) Morel-Voevodsky $\aone$-homotopy category, as constructed in \cite[\S 2 Theorem 3.2 and \S 3 Definition 2.1]{MV}.  Very briefly, one equips $\Spc_k$ with the $\aone$-local injective model structure (cofibrations are monomorphisms, weak equivalences are the $\aone$-local weak equivalences, and fibrations are determined by the right lifting property) and realizes $\ho{k}$ as the associated homotopy category.  More precisely, $\ho{k}$ is equivalent to the full-subcategory of $\hsnis$ consisting of $\aone$-local objects; this inclusion admits a left adjoint $L_{\aone}$ called the $\aone$-localization functor, which gives rise to the functor $\hsnis \to \ho{k}$ in Diagram \ref{eqn:motiviccategories}.  The functor $\Spc_k \to \ho{k}$ also sends a space $\mathscr{X}$ to its $\aone$-localization $L_{\aone}(\mathscr{X})$.

Given two spaces $\mathscr{X}$ and $\mathscr{Y}$, we write $[\mathscr{X},\mathscr{Y}]_{\aone}$ for the set of morphisms between the resulting objects in $\ho{k}$.  We write $\hop{k}$ for the pointed variant of $\ho{k}$ and as above, if it is not clear from context, to distinguish morphisms in this category we will explicitly specify the base-point.
\end{entry}

The passage from unstable to stable categories is by now standard: we replace spaces by spectra.  To do this in a fashion that gives a tensor structure that is commutative and associative on the nose, we use the theory of symmetric spectra.

\subsubsection*{Symmetric sequences}
Throughout this work $\mathcal{C}$ is assumed to be a symmetric monoidal category that is complete and cocomplete.  In our examples, $\mathcal{C}$ will be one of the following four examples.
\begin{itemize}
\item[i)] the category of simplicial sets $\Delta^{\circ}\Set$ equipped with the cartesian product, or the category $\Delta^{\circ}\Set_\bullet$ of pointed simplicial sets equipped with the smash product.
\item[ii)] the category of spaces $\Spc_k$ equipped with the cartesian product, or the category $\Spc_{k,\bullet}$ of pointed spaces equipped with the smash product.
\item[iii)] for a ``nice" abelian category $\mathcal{A}$,  the category $\Ch_k(\mathcal{A})$ of (bounded below) chain complexes (i.e., differential of degree $-1$) of objects in $\mathcal{A}$ equipped with the shuffle product of complexes.
\end{itemize}
The importance of the above choice of tensor product will be evident when we discuss the Dold-Kan correspondence below.

Everything we say below is a recapitulation of results from \cite{HoveyShipleySmith} or \cite{AyoubII}.

\begin{defn}[{\cite[Definition 2.1.1]{HoveyShipleySmith}}]
For each integer $n \geq 0$, define $\underline{n}$ to be $\emptyset$ if $n = 0$, and the set of integers $1 \leq i \leq n$ for $i > 0$.  Let $\mathfrak{S}$ be the groupoid whose objects are the sets $\underline{n}$ and whose morphisms are the bijections $\underline{n} \to \underline{n}$.  A {\em symmetric sequence in $\mathcal{C}$} is a functor $\mathfrak{S} \to \mathcal{C}$.
\end{defn}

\begin{notation}
Write $\Fun(\mathfrak{S},\mathcal{C})$ for the category of symmetric sequences in $\mathcal{C}$.
\end{notation}

\begin{lem}
If $\mathcal{C}$ is complete, cocomplete, or monoidal, so is $\Fun(\mathfrak{S},\mathcal{C})$, with all corresponding notions defined ``levelwise."
\end{lem}


The monoidal structure on symmetric sequences is explained in detail in, e.g., \cite[Definition 2.1.3 and Corollary 2.2.4]{HoveyShipleySmith} or \cite[\S 7.3]{CisinskiDeglise1}.  In particular, it makes sense to talk about a monoid $\mathscr{R}$ in $\mathscr{Fun}(\mathfrak{S},\mathcal{C})$, and of (left or right) modules over $\mathscr{R}$ \cite[\S VII.3-4]{MacLane}.  As a consequence, if $\mathscr{R}$ is a commutative monoid object in $\mathscr{Fun}(\mathfrak{S},\mathcal{C})$, given two $\mathscr{R}$-modules, we have an internal hom and tensor product satisfying the usual hom-tensor adjunction \cite[Lemma 2.2.8]{HoveyShipleySmith}.

If $\mathcal{C}'$ is another monoidal category, and $\Phi: \mathcal{C} \to \mathcal{C}'$ is a monoidal functor (see \cite[\S XI. 2]{MacLane}), then $\Phi$ induces a monoidal functor $\Fun(\mathfrak{S},\mathcal{C}) \to \Fun(\mathfrak{S},\mathcal{C}')$.

\subsubsection*{Symmetric spectra}
We write $S^n_s$ for the {\em simplicial $n$-sphere}, i.e., the constant sheaf in $\Spc_{k,\bullet}$ arising from the simplicial set $\Delta^n/\partial \Delta^n$.  The space $S^n_s$ has a natural action of the symmetric group $\Sigma_n$.  One defines a functor $\mathfrak{S} \to \Spc_{k,\bullet}$ by sending $\underline{n} \to S^n_s$ equipped with this action of $\Sigma_n$.  We write $\Sigma^{\infty}_sS^0_s$ for this symmetric sequence; it admits a natural structure of commutative ring object in $\mathscr{Fun}(\mathfrak{S},\Spc_{k,\bullet})$.  The subcategory of $\mathscr{Fun}(\mathfrak{S},\Spc_{k,\bullet})$ consisting of modules over $\Sigma^{\infty}_s S^0_s$ will be called the category of {\em symmetric spectra in $k$-spaces}; we write $\Sp^{\Sigma}(\Spc_k)$ for the resulting category.

\begin{ex}
Suppose $\mathscr{X}$ is a pointed $k$-space.  Define a functor $\mathfrak{S} \to \Spc_{k,\bullet}$ by $\underline{n} \mapsto S^n_s \wedge \mathscr{X}$ and where the action of $\Sigma_n$ is induced by the permutation on $S^n_s$ as explained above.  We write $\Sigma^{\infty}_s \mathscr{X}$ for this symmetric sequence.  As in \cite[Example 1.2.4]{HoveyShipleySmith}, one checks that $\Sigma^{\infty}_s \mathscr{X}$ is a symmetric spectrum in $k$-spaces that we call the {\em suspension symmetric spectrum of $\mathscr{X}$}.  In the sequel, we will refer to $\Sigma^{\infty}_s \mathscr{X}$ as the simplicial suspension spectrum of $\mathscr{X}$.
\end{ex}

\begin{ex}
If $\mathscr{X}$ is a $k$-space, write $\Z[\mathscr{X}]$ for the free sheaf of abelian groups on $\mathscr{X}$ and $\tilde{\Z}[\mathscr{X}]$ for the kernel of the morphism of sheaves $\Z[\mathscr{X}] \to \Z$ induced by the structure morphism $\mathscr{X} \to \Spec k$.  Define a functor $\mathfrak{S} \to \Spc_{k,\bullet}$ by $\underline{n} \mapsto \tilde{\Z}[S^n_s]$; write $\operatorname{H}\Z$ for this functor.  As in \cite[Example 1.2.5]{HoveyShipleySmith}, one checks that $\operatorname{H}\Z$ is a symmetric spectrum in $k$-spaces.  More generally, given a sheaf of abelian groups $\mathscr{M}$, we can define the Eilenberg-MacLane spectrum $\operatorname{H}\mathscr{M}$.
\end{ex}

\subsubsection*{Stable homotopy categories}
We now discuss the various stable homotopy categories that arise in our discussion.  We want our theory to have well-defined internal hom and tensor products, and so we follow the now standard constructions of stable $\aone$-homotopy theory.  The category of symmetric spectra in $k$-spaces can also be viewed as the category of Nisnevich sheaves of symmetric spectra.

The category of ordinary symmetric spectra has the structure of a monoidal model category (for the definition of monoidal model category see \cite[Definition 3.1]{SchwedeShipley}).  There are several model structures one naturally considers.  For the stable model structure, the weak equivalences, cofibrations and fibrations are given in \cite[Definitions 3.1.3, 3.4.1, and 3.4.3]{HoveyShipleySmith} and \cite[Theorem 3.4.4]{HoveyShipleySmith} establishes that these actually determine a model structure and {\em ibid.} Corollary 5.3.8 establishes that this model structure is actually monoidal.

\begin{entry}[$S^1$-stable simplicial homotopy category]
Ayoub explains how the category of Nisnevich sheaves with values in a monoidal model category can naturally be equipped with a monoidal model structure; see \cite[Definition 4.4.40, Corollary 4.4.42 and Proposition 4.4.62]{AyoubII}.  Applying this construction to the monoidal model category of symmetric spectra shows that $\Sp^{\Sigma}(\Spc_k)$ has the structure of a monoidal model category; we write $\SH_{S^1}(k)$ for the homotopy category of this model category and refer to it as the stable homotopy category of $S^1$-spectra.  By construction the functor sending a space $\mathscr{X}$ to the symmetric spectrum $\Sigma^{\infty}_s\mathscr{X}_+$ factors through $\hsnis$ inducing the functor $\hsnis \to \SH_{S^1}(k)$ of Diagram \ref{eqn:motiviccategories}.
\end{entry}

\begin{entry}[$S^1$-stable $\aone$-homotopy category]
The category $\SH_s(k)$ of Diagram \ref{eqn:motiviccategories}, called the stable $\aone$-homotopy category of $S^1$-spectra, is obtained from $\SH_{S^1}(k)$ by the procedure of Bousfield localization.  The first construction is due to Jardine (see \cite[Theorems 4.3.2 and 4.3.8]{JardineSymmetric} for two different model structures).  Ayoub gives an equivalent presentation.  The category $\Sp^{\Sigma}(\Spc_k)$ can be equipped with an $\aone$-local model structure; see \cite[Definition 4.5.12]{AyoubII}.  By {\em ibid.} Proposition 4.2.76, it follows that the resulting model structure is again a symmetric monoidal model structure. The homotopy category of this model structure is $\SH_{s}(k)$.  The homotopy category so constructed is equivalent to the one described in \cite[Definition 4.1.1]{MStable} by \cite[Theorem 4.40]{JardineSymmetric}.

One can construct an $\aone$-resolution functor that commutes with finite products using the Godement resolution functor and the singular construction $Sing_*^{\aone}$ of \cite[\S 2 Theorem 1.66 and p. 88]{MV}.  Using this $\aone$-resolution functor we can speak of symmetric sequences of $\aone$-local objects.  One defines the $\aone$-local symmetric sphere spectrum by taking the functor $\underline{n} \mapsto L_{\aone}(S^n_s)$ equipped with the induced action of the symmetric groups. We then consider the category of modules over the $\aone$-local symmetric sphere spectrum; by construction this is a model of $\Sp^{\Sigma}(\Spc_k)$.  In an analogous fashion, one defines the $\aone$-local symmetric suspension spectrum of a pointed space $(\mathscr{X},x)$ as the symmetric sequence $\underline{n} \mapsto L_{\aone}(S^n_s \wedge \mathscr{X})$ equipped with the induced actions of the symmetric groups.  The functor $\Spc_{k} \to \SH_{s}(k)$ is induced by the functor sending a space $\mathscr{X}$ to the $\aone$-local symmetric suspension spectrum of $\mathscr{X}_+$.
\end{entry}

\begin{defn}
\label{defn:s1stableaonehomotopysheaves}
Suppose $\mathscr{E}$ is an $\aone$-local symmetric spectrum in $k$-spaces.  The $i$-th $S^1$-stable $\aone$-homotopy sheaf of $\mathscr{E}$, denoted $\bpi_i^{s}(\mathscr{E})$, is the Nisnevich sheaf on $\Sm_k$ associated with the presheaf
\[
U \mapsto \hom_{\SH_{s}(k)}(S^i_s \wedge \Sigma^{\infty}_s U_+,\mathscr{E}).
\]
\end{defn}

The main structural property of these sheaves $\bpi_i^{s}(\mathscr{X})$ is summarized in the following result, which will be used without mention in the sequel.

\begin{prop}[{\cite[Theorem 6.1.8 and Corollary 6.2.9]{MStable}}]
If $\mathscr{E}$ is an $\aone$-local symmetric $S^1$-spectrum in $k$-spaces, the sheaves $\bpi_i^{s}(\mathscr{X})$ are strictly $\aone$-invariant.
\end{prop}

\begin{entry}[$\pone$-stable $\aone$-homotopy category]
The category $\SH(k)$ is the stable $\aone$-homotopy category of $\pone$-spectra (see, e.g., \cite[\S 5]{MIntro}).  This category was also constructed in \cite[Theorem 4.2]{JardineSymmetric}, but again we follow Ayoub for consistency.  Now, instead of considering symmetric $S^1$-spectra, one considers symmetric $\pone$ or $T$-spectra.  In other words, we consider the category $\mathscr{Fun}(\mathfrak{S},\Spc_{k,\bullet})$ of symmetric sequences.  Consider $\pone$ is a pointed space with base point $\infty$.  Equip ${\pone}^{\wedge n}$ with an action of $\Sigma_n$ by permutation of the factors.  The assignment $\underline{n} \mapsto {\pone}^{\wedge n}$ determines a symmetric sequence.  Moreover, this object has the structure of a commutative ring object in the category of symmetric sequences (this object is the ``free" symmetric spectrum generated by $(\pone,\infty)$); this spectrum, denoted $\mathbb{S}^0_k$, is the {\em sphere symmetric $\pone$-spectrum} or just {\em sphere spectrum} if no confusion can arise.  A symmetric $\pone$-spectrum is a symmetric sequence having the structure of a module over $\mathbb{S}^0_k$.  Write $\Sp_{\pone}^{\Sigma}(\Spc_k)$ for the full subcategory of $\mathscr{Fun}(\mathfrak{S},\Spc_{k,\bullet})$ consisting of symmetric $\pone$-spectra.

One equips the category $\Sp_{\pone}^{\Sigma}(\Spc_k)$ with a model structure as in \cite[Definition 4.5.21]{AyoubII}.  The category $\SH(k)$ is the homotopy category of this model structure.  By \cite[Theorem 4.31]{JardineSymmetric}, the category so constructed is equivalent to the category of $\pone$-spectra considered by Voevodsky in \cite{VICM} and Morel in \cite[\S 5]{MIntro}.  Given a pointed space $(\mathscr{X},x)$, the suspension symmetric $\pone$-spectrum, denoted $\Sigma^{\infty}_{\pone}\mathscr{X}$ is given by the functor $\underline{n} \mapsto {\pone}^{\wedge n} \wedge \mathscr{X}$ with the symmetric group acting by permuting the first $n$-factors.  We write $\mathbb{S}^i_k$ for the suspension symmetric $\pone$-spectrum of the {\em simplicial} $i$-sphere $S^i_s$.
\end{entry}

\begin{defn}
\label{defn:stableaonehomotopysheaves}
Suppose $\mathscr{E}$ is a symmetric $\pone$-spectrum.  The $i$-th stable $\aone$-homotopy sheaf of $\mathscr{E}$, denoted $\bpi_i^{s}(\mathscr{E})$ is the Nisnevich sheaf on $\Sm_k$ associated with the presheaf
\[
U \mapsto \hom_{\SH(k)}(\mathbb{S}^i_k \wedge \Sigma^{\infty}_{\pone}U_+,\mathscr{E}).
\]
\end{defn}

Strict ring and module structures at the level of spaces induce ring and module structures at the level of homotopy categories and thus we deduce the following result.

\begin{prop}
\label{prop:modulestructureonstableaonehomotopysheaves}
If $\mathscr{E}$ is a symmetric $\pone$-spectrum, the sheaf $\bpi_0^s(\mathscr{E})$ is a sheaf of modules over the sheaf of rings $\bpi_0^s(\mathbb{S}^0_k)$; a morphism $f: \mathscr{E} \to \mathscr{E}'$ of symmetric $\pone$-spectra induces a morphism of $\bpi_0^s(\mathbb{S}^0_k)$-modules.
\end{prop}

The category $\SH(k)$ can be viewed as $\SH_{s}(k)$ with  the ($\aone$-localized) suspension spectrum of $\gm$ formally inverted.  More precisely, we have the following result.

\begin{prop}
\label{prop:stablehomotopycolimit}
For any smooth scheme $U$, and any pointed space $(\mathscr{X},x)$ the canonical morphism
\[
\colim_n \hom_{\SH_{s}(k)}(\Sigma^{\infty}_s \gm^{\wedge n} \wedge \Sigma^{\infty}_s(U_+),\Sigma^{\infty}_s \gm^{\wedge n} \wedge \Sigma^{\infty}_s \mathscr{X}) \longrightarrow \hom_{\SH(k)}(\Sigma^{\infty}_{\pone}U_+,\Sigma^{\infty}_{\pone}\mathscr{X})
\]
is an isomorphism.
\end{prop}

\begin{proof}
Ayoub shows that the cyclic permutation $(123)$ acts as the identity on $T^{\wedge 3}$ in \cite[4.5.65]{AyoubII}.  The result then follows from \cite[Theorems 4.3.61 and 4.3.79]{AyoubII}.
\end{proof}

\begin{rem}
One useful Corollary to this Proposition is that the stable $\aone$-homotopy sheaves $\bpi_i^{s}(\Sigma^{\infty}_{\pone}\mathscr{X}_+)$ of a $k$-space $\mathscr{X}$ are always strictly $\aone$-invariant since they are colimits of strictly $\aone$-invariant sheaves.
\end{rem}

\subsubsection*{Stable $\aone$-homotopy sheaves of spheres}
\begin{defn}
\label{defn:stableaonehomotopygroupsofspheres}
For every integer $n \in \Z$ we set
\[
\K^{MW}_n := \bpi_0^{s}(\Sigma^{\infty}_{\pone}\gm^{\wedge n})
\]
\end{defn}

\begin{rem}
\label{rem:morelscomputation}
Since stable $\aone$-homotopy groups are strictly $\aone$-invariant, the sheaves $\K^{MW}_n$ are strictly $\aone$-invariant.  What we refer to as ``Morel's computation of the zeroth stable $\aone$-homotopy sheaves" is a description of the sheaves $\K^{MW}_n$ in concrete terms.  More precisely, the sections of $\K^{MW}_n$ over fields are precisely the Milnor-Witt K-theory groups defined by Hopkins and Morel (see \cite[\S 6.3]{MIntro}).  In \cite[\S 2.1-2]{MField}, Morel explains how to define these sheaves in an ``elementary" fashion in terms of residue maps for Milnor-Witt K-theory.  Definition \ref{defn:stableaonehomotopygroupsofspheres} will be used in a more concrete fashion by means of Morel's computations in Theorem \ref{thm:morelfaselcomparison}.
\end{rem}

\subsubsection*{The sheaf-theoretic Dold-Kan correspondence}
Let $\Delta^{\circ}\Ab_k$ denote the category of simplicial Nisnevich sheaves of abelian groups.  Let $\Ch_{\geq 0}\Ab_k$ denote the category of chain complexes (differential of degree $-1$) of Nisnevich sheaves of abelian groups situated in homological degree $\geq 0$.  The functor of normalized chain complex determines a functor
\[
N: \Delta^{\circ}\Ab_k \longrightarrow \Ch_{\geq 0}\Ab_k.
\]
One can also construct a functor $K: \Ch_{\geq 0}\Ab_k \to \Delta^{\circ}\Ab_k$.  The sheaf-theoretic Dold-Kan correspondence shows that the functors $N$ and $K$ give mutually inverse equivalences of categories.  Because $N$ and $K$ are mutually inverse, they are also adjoint to each other on both sides; this observation is important if we consider monoidal structures on the categories in question.  We view $N$ as the right adjoint.  Then, $N$ is a (lax) symmetric monoidal functor with respect to the shuffle product, though not with respect to the ``usual" tensor product on chain complexes, and it sends the unit object in the category of simplicial abelian groups to the unit object in the category of chain complexes of abelian groups.

\subsubsection*{$\aone$-derived categories}
We now describe the categories in the second to last column of Diagram \ref{eqn:motiviccategories}.  We follow the presentation of Cisinski-D\'eglise \cite[\S 7]{CisinskiDeglise1} or \cite[\S 5.3]{CisinskiDeglise2}.  Write $\Ch\Ab_k$ for the category of bounded below chain complexes (i.e., differential of degree $-1$) of Nisnevich sheaves of abelian groups. Begin by considering the space of symmetric sequences $\mathscr{Fun}({\mathfrak{S}},\Ch\Ab_k)$.

There is an obvious inclusion functor $\Ch_{\geq 0}\Ab_k \hookrightarrow \Ch\Ab_k$, and we extend $N$ from the previous section to a functor $N: \Delta^{\circ}\Ab_k \longrightarrow \Ch\Ab_k.$  The functor $N$ is still (lax) monoidal.  As a consequence, the induced composite functor
\[
\tilde{N}: \Fun(\mathfrak{S},\Delta^{\circ}\Ab_k) \to \Fun(\mathfrak{S}, \Ch\Ab_k)
\]
is also (lax) monoidal.

\begin{entry}[Derived category of Nisnevich sheaves of abelian groups]
The Eilenberg-MacLane symmetric spectrum $\operatorname{H}\Z \in \Fun(\mathfrak{S},\mathscr{Spc}_{k,\bullet})$ is naturally an object in $\Fun(\mathfrak{S},\Delta^{\circ}\Ab_k)$ and by composition with $\tilde{N}$ we get a symmetric sequence $\tilde{N}(\operatorname{H}\Z)$.  Since $\operatorname{H}\Z$ is a ring object, and the functor $\tilde{N}$ is (lax) monoidal, it follows that $\tilde{N}(\operatorname{H}\Z)$ is a ring object in $\Fun(\mathscr{S},\Ch\Ab_k)$.  We let $\Sp^{\Sigma}(\Ch\Ab_k)$ denote the full subcategory of $\Fun(\mathscr{S},\Ch\Ab_k)$ consisting of modules over $\tilde{N}(\operatorname{H}\Z)$.  One equips $\Sp^{\Sigma}(\Ch\Ab_k)$ with a level-wise model structure \cite[\S 7.6 and Proposition 7.9]{CisinskiDeglise1}; this model structure is furthermore monoidal, and we write $\operatorname{D}_{-}(\Ab_k)$ for the resulting homotopy category.
\end{entry}

\begin{entry}
The functor $N$ in the Dold-Kan correspondence is part of a Quillen equivalence of model categories.  One also has the (monoidal) free abelian group functor $\Z(\cdot): \Spc_k \to \Delta^{\circ}\Ab_k$.  The composite of these two functors gives a functor $\Fun(\mathfrak{S},\Spc_{k,\bullet}) \to \Fun(\mathfrak{S},\Ch\Ab_k)$.  By construction, the image of the sphere symmetric sequence under this composite functor is $\tilde{N}(\operatorname{H}\Z)$, and one obtains a functor between corresponding categories of symmetric spectra:
\begin{equation}
\label{eqn:simplicialhurewicz}
\Sp^{\Sigma}(\Spc_k) \to \Sp^{\Sigma}(\Ch\Ab_k).
\end{equation}
By \cite[Theorem 9.3]{Hovey}, this functor induces a Quillen functor
\[
\SH_{S^1}(k) \to \operatorname{D}(\Ab_k)
\]
that we refer to alternatively as a Hurewicz functor or the derived functor of abelianization.  In any case, we write $\Z[n]$ for the image of $S^n_s$ under this functor.
\end{entry}

\begin{entry}[Effective $\aone$-derived category]
The category $\Daone^{\eff}(k)$ of Diagram \ref{eqn:motiviccategories} is called the effective $\aone$-derived category.  This category is constructed from $\Sp^{\Sigma}(\Ch\Ab_k)$ by $\aone$-localization.  Thus, one gets an $\aone$-local model structure on $\Sp^{\Sigma}(\Ch\Ab_k)$; again this model structure is monoidal, and we write $\Daone^{\eff}(k)$ for the resulting homotopy category.  Cisinski and D\'eglise prove (see \cite[Proposition 5.3.19 and 5.3.20]{CisinskiDeglise2}) that the homotopy category of this model structure is equivalent to the $\aone$-derived category as defined in \cite[\S 3.2]{MField}.  Again, one constructs an $\aone$-localization functor, which can be assumed to commute with finite products.
\end{entry}

\begin{entry}
It is not true that Functor \ref{eqn:simplicialhurewicz} takes $\aone$-local objects to $\aone$-local objects.  Nevertheless, if $(\mathscr{X},x)$ is a pointed space, we can consider the suspension symmetric spectrum $\Sigma^{\infty}_s \mathscr{X}$.  Applying Functor \ref{eqn:simplicialhurewicz} we get an object in $\Sp^{\Sigma}(\Ch\Ab_k)$ that we can $\aone$-localize.  We write $\tilde{C}_*^{\aone}(\mathscr{X})$ for the resulting object.  On the other hand, we could $\aone$-localize $\Sigma^{\infty}_s \mathscr{X}$ and then apply Functor \ref{eqn:simplicialhurewicz}.  There is a natural morphism from the latter to the former, and since all functors in question preserve finite products, this functor is monoidal.  As a consequence, there is an induced functor $\SH_{S^1}(k) \to \Daone^{\eff}(k)$.  If $\mathscr{X}$ is not pointed, we write $C_*^{\aone}(\mathscr{X})$ for $\tilde{C}_*^{\aone}(\mathscr{X}_+)$.  If $(\mathscr{X},x)$ is a pointed space, we write $\tilde{C}_*^{\aone}(\mathscr{X})[n]$ for the tensor product $\tilde{C}_*^{\aone}(\mathscr{X}) \tensor \Z[n]$.
\end{entry}

\begin{defn}
\label{defn:unstableaonehomologysheaves}
The $i$-th $\aone$-homology sheaf of a space $\mathscr{X}$, denoted $\H_i^{\aone}(\mathscr{X})$, is the Nisnevich sheaf on $\Sm_k$ associated with the presheaf
\[
U \mapsto \hom_{\Daone^{\eff}(k)}(C_*^{\aone}(U)[n],C_*^{\aone}(\mathscr{X})).
\]
\end{defn}

In the category $\Daone^{\eff}(k)$, the shift functor coincides with tensoring with $\Z[1]$.  Fix a rational point on $\pone$ and call it $\infty$.  Point $\gm$ by $1$; the open cover of $\pone$ by two copies of $\aone$ with intersection $\gm$ gives rise to an isomorphism $S^1_s \wedge \gm \isomt \pone$ in $\ho{k}$.  This gives an identification $\tilde{C}_*^{\aone}(S^1_s \wedge \gm) \isomt \tilde{C}_*^{\aone}(\pone)$.  Since the functor from the $S^1$-stable $\aone$-homotopy category to the effective $\aone$-derived category is monoidal, we may identify this object with $\tilde{C}_*^{\aone}(\gm)[-1]$.

\begin{defn}
\label{defn:enhancedTatecomplex}
The enhanced Tate complex, denoted $\Z \langle 1 \rangle$, is the object $\tilde{C}_*^{\aone}(\pone)[-2]$.
\end{defn}

\begin{entry}[Stable $\aone$-derived category]
Next, we define the category $\Daone(k)$---the stable $\aone$-derived category---from Diagram \ref{eqn:motiviccategories}.  The category $\Daone(k)$ is obtained from $\Daone^{\eff}(k)$ by formally inverting the enhanced Tate complex.  Again, this task is accomplished by means of the theory of symmetric spectra.  The simplest thing to do, ensuring compatibility with the construction of the stable $\aone$-homotopy category given above is to proceed as follows.  Consider the $\aone$-localization of the normalized chain complex of the free abelian group on the sphere symmetric $\pone$-spectrum.  This determines a monoid in the category of symmetric sequences, and we can consider the full subcategory, denoted $\Sp^\Sigma_{\pone}(\Ch\Ab_k)$, of $\Fun(\mathfrak{S},\Ch\Ab_k)$ consisting of modules over this object.  As above, we equip this category with a monoidal model structure and let $\Daone(k)$ denote the associated homotopy category.

The symmetric $\pone$-suspension spectrum of a pointed space $(\mathscr{X},x)$ defines an object $\tilde{C}_*^{s\aone}(\mathscr{X})$ of $\Sp^\Sigma_{\pone}(\Ch\Ab_k)$: take the $\aone$-localization of the normalized chain complex of the free abelian group on the suspension symmetric $\pone$-spectrum of $\mathscr{X}$.  We call this category the category of symmetric $\pone$-chain complexes.  Write $\1_k$ for the unit object in $\Daone(k)$ for the induced tensor structure, i.e., the complex $\tilde{C}_*^{s\aone}(S^0_k)$, and write $\1_k[n]$ for the space $\tilde{C}_*^{s\aone}(S^n_s)$, and $\tilde{C}_*^{s\aone}(\mathscr{X})[n]$ for the space $\tilde{C}_*^{s\aone}(\mathscr{X}) \tensor \1_k[n]$.  If $\mathscr{X}$ is not pointed, write $C_*^{s\aone}(\mathscr{X})$ for $\tilde{C}_*^{s\aone}(\mathscr{X}_+)$.
\end{entry}

\begin{defn}
\label{defn:stableaonehomologysheaves}
For any object $\mathscr{X}$ of $\Sp^\Sigma_{\pone}(\Ch\Ab_k)$, the $i$-th stable $\aone$-homology sheaf $\H_i^{s\aone}(\mathscr{X})$ is the sheaf on $\Sm_k$ associated with the presheaf
\[
U \mapsto \hom_{\Daone(k)}(C_*^{s\aone}(U)[i],C_*^{s\aone}(\mathscr{X})).
\]
\end{defn}

\begin{prop}
\label{prop:modulestructureonstableaonehomology}
If $\mathscr{X}$ is an object of $\Sp^\Sigma_{\pone}(\Ch\Ab_k)$, the sheaf $\H_0^{s\aone}(\mathscr{X})$ is a sheaf of (left or right) modules over the sheaf of rings $\H_0^{s\aone}(\1_k)$; a morphism $f: \mathscr{X} \to \mathscr{Y}$ of objects in $\Sp^\Sigma_{\pone}(\Ch\Ab_k)$ induces a morphism of $\H_0^{s\aone}(\1_k)$-modules.
\end{prop}

We would like to know that morphisms in the category so defined can actually be obtained by inverting $\Z \langle 1 \rangle$ in $\Daone^{\eff}(k)$.

\begin{prop}
\label{prop:stablederivedcolimit}
For any smooth scheme $U$, any $i$ and any pointed space $(\mathscr{X},x)$ the canonical morphism
\[
\colim_n \hom_{\Daone^{\eff}(k)}(C_*^{\aone}(U)\langle n \rangle[i],\tilde{C}_*^{\aone}(\mathscr{X})\langle n \rangle[i]) \longrightarrow \hom_{\Daone(k)}(C_*^{s\aone}(U),\tilde{C}_*^{s\aone}(\mathscr{X}))
\]
is an isomorphism.
\end{prop}

\begin{proof}
The proof is identical to the corresponding statement in the stable $\aone$-homotopy category, i.e., Proposition \ref{prop:stablehomotopycolimit}.
\end{proof}

\subsubsection*{Hurewicz functors}
Composing the free abelian group functor and the normalized chain complex functor with $\aone$-localization, there are functors
\[
\begin{split}
\Sp^\Sigma(\Spc_k) &\longrightarrow \Sp^{\Sigma}(\Ch\Ab_k), \text{ and } \\
\Sp^\Sigma_{\pone}(\Spc_k) &\longrightarrow \Sp^\Sigma_{\pone}(\Ch\Ab_k).
\end{split}
\]
By construction, both functors are monoidal and induce functors (which we refer to as {\em Hurewicz functors}) on homotopy categories
\[
\begin{split}
\SH_{s}(k) &\longrightarrow \Daone^{\eff}(k), \text{ and }\\
\SH(k) &\longrightarrow \Daone(k).
\end{split}
\]
Consequently, if $\mathscr{X}$ is a space, there are induced morphisms
\begin{equation}
\label{eqn:hurewiczhomomorphism}
\begin{split}
\bpi_i^{s}(\Sigma^{\infty}_{s}\mathscr{X}_+) &\longrightarrow \H_i^{\aone}(\mathscr{X}), \text{ and }\\
\bpi_i^{s}(\Sigma^{\infty}_{\pone}\mathscr{X}_+) &\longrightarrow \H_i^{s\aone}(\mathscr{X})
\end{split}
\end{equation}
that we refer to as {\em Hurewicz morphisms}.

\begin{prop}
\label{prop:hurewiczcommuteswithmodulestructures}
If $\mathscr{X}$ is any $k$-space, the following diagram commutes:
\[
\xymatrix{
\bpi_0^{s}(\mathbb{S}^0_k) \tensor \bpi_0^s(\Sigma^{\infty}_{\pone}\mathscr{X}_+) \ar[r]\ar[d] & \bpi_0^s(\Sigma^{\infty}_{\pone}\mathscr{X}_+) \ar[d]\\
\H_0^{s\aone}(\1_k) \tensor \H_0^{s\aone}(\mathscr{X}) \ar[r] & \H_0^{s\aone}(\mathscr{X}).
}
\]
\end{prop}

\begin{proof}
The Hurewicz functor is induced by a monoidal functor.  The result then follows immediately by combining Proposition \ref{prop:modulestructureonstableaonehomotopysheaves} and \ref{prop:modulestructureonstableaonehomology}.
\end{proof}

\subsubsection*{Motivic categories}
\begin{entry}[Derived category of correspondences]
Let $\Cor_k$ denote the category whose objects are smooth $k$-schemes and where the set of morphisms from $X$ to $Y$ is the free abelian group on integral closed subschemes of $X\times_k Y$  finite and surjective over a component of $X$.  Sending a morphism $f: X \to Y$ of smooth $k$-schemes to its graph determines a functor $\Sm_k \to \Cor_k$.  Precomposing with this functor, any (say abelian group valued) presheaf on $\Cor_k$ can be viewed as a presheaf on $\Sm_k$.  Write $\Ab_{\Nis}(\Cor_k)$ for the full subcategory of abelian group valued presheaves on $\Cor_k$ that are Nisnevich sheaves.  Write $\Z_{tr}(X)$ for the representable presheaf $U \mapsto \Cor_k(U,X)$; this presheaf can be shown to be a Nisnevich sheaf on $\Sm_k$.  Abusing notation, we write $D_{-}(\Cor_k)$ for the derived category of the abelian category $\Ab_{\Nis}(\Cor_k)$.  This category can also be constructed using symmetric spectra.
\end{entry}

\begin{entry}[Effective motivic complexes]
We let $\DM^{\eff}(k,R)$ be Voevodsky's triangulated category of effective motivic complexes; when $R = \Z$ we suppress it from the notation and write $\DM^{\eff}(k)$.  The category $\DM^{\eff}(k,R)$ is the $\aone$-localization of the derived category of Nisnevich sheaves of $R$-modules with transfers.  We write $\M(X)$ for the complex $C_*R_{tr}(X)$ viewed as an object of $\DM^{\eff}(k,R)$.  The category $\DM^{\eff}(k,R)$ can be viewed as the subcategory of the derived category of Nisnevich sheaves of $R$-modules with transfers consisting of $\aone$-local objects.  The functor forget transfers preserves $\aone$-local objects and therefore induces a functor
\[
\DM^{\eff}(k,R) \longrightarrow \Daone^{\eff}(k,R).
\]
The next result shows that this functor is part of an adjunction of model categories.
\end{entry}

\begin{prop}
\label{prop:aonederivedmotivicadjunction}
There is an adjunction
\[
\xymatrix{
\Daone^{\eff}(k,R) \ar@<.4ex>[r] & \ar@<.4ex>[l] \DM^{\eff}(k,R)
}
\]
The adjoint functor preserves tensor products.
\end{prop}

\begin{proof}
See \cite[10.4.1]{CisinskiDeglise2}
\end{proof}

\begin{defn}
Assume $k$ is a field and $X$ is a smooth $k$-variety.  The $0$-th Suslin homology sheaf $\H_0^S(X)$ is the $0$-th homology sheaf of the complex $L_{\aone}C_*\Z_{tr}(X)$.
\end{defn}


\begin{prop}\label{prop:stabletosuslinhomology}
Suppose $X$ is a smooth $k$-scheme and $k$ is a perfect field.  The canonical map $\H_0^{\aone}(X) \to \H_0^S(X)$ factors through a morphism $\H_0^{s\aone}(X) \to \H_0^S(X)$.
\end{prop}

\begin{proof}
Use the functor of Proposition \ref{prop:aonederivedmotivicadjunction}.  By construction, given a smooth scheme $X$, this functor sends $C_*^{\aone}(X)$ to $\M(X)$ and is compatible with shifts.  In particular, $\Z \langle 1 \rangle$ is sent to $\Z(1)$.  We can identify $\H_0^{\aone}(X)$ with $\hom_{\Daone^{\eff}(k)}(\Z,C_*^{\aone}(X))$.  For every integer $i \geq 0$, we therefore have functorial maps:
\begin{equation}
\label{eqn:stabilizationcomparison}
\hom_{\Daone^{\eff}(k)}(\Z\langle i \rangle,C_*^{\aone}(X)\langle i \rangle) \longrightarrow \hom_{\DM^{\eff}(k)}(\Z(i),\M(X)(i)).
\end{equation}
By Voevodsky's cancelation theorem, we know that the canonical map
\[
\hom_{\DM^{\eff}(k)}(\Z(j),\M(X)(j)) \longrightarrow \hom_{\DM^{\eff}(k)}(\Z(j+1),\M(X)(j+1))
\]
is an isomorphism for any $j \geq 0$.  In particular, we get an isomorphism
\[
\H_0^{S}(X) \isomt \hom_{\DM^{\eff}(k)}(\Z(i),\M(X)(i))
\]
for any $i \geq 0$.  In light of Proposition \ref{prop:stablederivedcolimit}, the colimit of the maps of Formula \ref{eqn:stabilizationcomparison} followed by the inverse to the isomorphism of the previous line gives the morphism of the statement.
\end{proof}

\subsection{Suslin homology and $0$-cycles of degree $1$}
\label{ss:suslinhomology}
As a model for our study of the $0$-th $\aone$-homology sheaf, in this section we review the relationship between the $0$-th Suslin homology sheaf and $0$-cycles of degree $1$.

\begin{lem}
\label{lem:suslinchow}
Assume $k$ is a perfect field, and $X$ is a smooth proper $k$-variety.  For any separable finitely generated extension $L/k$, we have a canonical identification $\H_0^S(X)(L) = CH_0(X_L)$.
\end{lem}

\begin{proof}
We have the following sequence of identifications.  Let $n = \dim X$.  For simplicity of notation, we write $H_0^S(X_L)$ for $\H_0^S(X)(L)$.
\[
\begin{split}
H_0^S(X_L) &\stackrel{(a)}{=} \hom_{\DM^{\eff}_k}(\Z,\M(X_L)) \\
&\stackrel{(b)}{=}  \hom_{\DM^{\eff}_L}(\Z,\M(X_L)) \\
&\stackrel{(c)}{\cong}  \hom_{\DM_L}(\Z,\M(X_L)) \\
&\stackrel{(d)}{\cong} \hom_{\DM_L}(\M(X_L)^*,\Z) \\
&\stackrel{(e)}{\cong} \hom_{\DM_L}(\M(X_L)(-n)[-2n],\Z) \\
&\stackrel{(f)}{\cong} \hom_{\DM^{\eff}_L}(\M(X_L),\Z(n)[2n]) \\
&\stackrel{(g)}{=} {\mathbb H}^{2n}_{\Nis}(X_L,\Z(n)) \\
&\stackrel{(h)}{\cong} H^n_{\Nis}(X_L,\K^M_n) \\
&\stackrel{(i)}{=} CH_0(X_L). \\
\end{split}
\]
Identification (a) follows from the definition of $\M(X)$.  Identification (b) follows from, e.g., \cite[Exercise 1.12]{MVW}.  That (c) is an isomorphism is the statement of Voevodsky's cancelation theorem \cite[Corollary 4.10]{VCancellation}, which implies that homomorphisms in $\DM^{\eff}_L$ and $\DM_L$ coincide.  Isomorphism (d) follows from the fact that the category $\DM_L$ admits duals.  Isomorphism (e) is a consequence of the explicit description of duals in $\DM_k$; this can be viewed as a form of Atiyah duality (we will explain this in detail later).  For smooth varieties, Friedlander-Voevodsky proved a duality theorem under the assumption that $k$ admits resolution of singularities (see, e.g., \cite[Theorem 16.27]{MVW} for some recollections).  However, the assumption that $X$ is, in addition, proper allows us to weaken the hypothesis on $k$ to merely assuming it is perfect.

Isomorphism (f) follows again from Voevodsky's cancelation theorem.  Identification (g) follows from the fact that $\Z(n)[2n]$ is $\aone$-local; see \cite[Corollary 14.9]{MVW}.  Identification (h) requires more argument.  One knows \cite[Theorem 5.1]{MVW} that there is a canonical isomorphism $\underline{H}^n(\Z(n))(L) = \K^M_n(L)$ (this is essentially the Nesterenko-Suslin-Totaro theorem but the aforementioned proof is self-contained).  Moreover, this argument shows that $\underline{H}^i(\Z(n))(L)$ vanishes for $i > n$.  The hypercohomology spectral sequence together with dimensional vanishing give the required identification. Identification (i) is a result of Kato \cite{KatoMilnor} and is a consequence of the existence of the Gersten resolution for the sheaf $\K^M_n$.
\end{proof}

\begin{lem}
\label{lem:suslindegree}
Assume $k$ is a perfect field, and $X$ is a smooth proper $k$-variety.  The pushforward map $\H_0^S(X) \to \H_0^S(\Spec k) = \Z$ coincides with the degree map $CH_0(X_L) \to \Z$ upon evaluation on sections over finitely generated extensions $L/k$.
\end{lem}

\begin{proof}
This follows from the explicit construction of the duality map; see \cite[Theorem 2.11]{VMilnor} together with the identifications of Lemma \ref{lem:suslinchow}.
\end{proof}

\begin{cor}
Assume $k$ is a perfect field, and $X$ is a smooth proper $k$-variety.  There is a canonical bijection between splittings $s: \Z \to \H_0^S(X)$ of the degree homomorphism and $0$-cycles of degree $1$ given by $s \mapsto s(1)$.
\end{cor}

\begin{proof}
Write $\Ab^{\aone}_{k,tr}$ for the category of strictly $\aone$-invariant Nisnevich sheaves with transfers.  By the Yoneda lemma, for any field $k$, there is a canonical bijection
\[
\H_0^S(X)(k) = \hom_{\Spc_k}(\Spec k,\H_0^{S}(X)) \isomto \hom_{\Ab^{\aone}_{k,tr}}(\Z,\H_0^S(X));
\]
the map sends an element $x \in \H_0^{S}(X)(k)$ to the free abelian group generated by multiples of $x$.  The result is then a consequence of Lemma \ref{lem:suslindegree}.
\end{proof}

\subsection{Abelianization}
\label{ss:abelianization}
The goal of this section is to identify the $0$-th stable $\aone$-homotopy sheaf of a smooth scheme with an analogous object of ``homological" nature, but in a manner compatible with the all the additional structure.

\subsubsection*{The $S^1$-stable story}
The following result was stated as \cite[Theorem 4.3.2]{Morelpi0}, but a proof was not given there. We write $\Ab^{\aone}_k$ for the category of strictly $\aone$-invariant sheaves of abelian groups.

\begin{lem}
\label{lem:simplicialhurewicz}
Assume $k$ is a field, and suppose $\mathscr{X}$ is a space.  The Hurewicz homomorphism $\bpi_0^{s}(\Sigma^{\infty}_s \mathscr{X}_+) \to \H_0^{\aone}(\mathscr{X})$ of \textup{Formula \ref{eqn:hurewiczhomomorphism}} is an isomorphism.
\end{lem}

\begin{proof}
Suppose $\mathscr{M}$ is a strictly $\aone$-invariant sheaf of abelian groups, and let $\mathrm{H}\mathscr{M}$ denote the associated Eilenberg-MacLane $S^1$-spectrum.  The $S^1$-suspension spectrum $\Sigma^{\infty}_{s}\mathscr{X}_+$ is $(-1)$-connected and hence its $\aone$-localization is again $(-1)$-connected by the stable $\aone$-connectivity theorem \cite[Theorem 6.1.8]{MStable}. Using the fact that $\mathrm{H}\mathscr{M}$ is also $(-1)$-connected, existence and functoriality of the Postnikov tower gives rise to a canonical isomorphism
\[
\hom_{\Ab^{\aone}_k}(\bpi_0^{s}(\Sigma^{\infty}_s \mathscr{X}),\bpi_0^{s}(\mathrm{H}\mathscr{M})) \isomto H^0_{\Nis}(\mathscr{X},\mathscr{M}).
\]
The analogous construction in $\Daone^{\eff}(k)$ gives rise to an isomorphism
\[
\hom_{\Ab^{\aone}_k}(\H_0^{\aone}(\mathscr{X}),\mathscr{M}) \isomto H^0_{\Nis}(\mathscr{X},\mathscr{M}).
\]
The result follows immediately from the Yoneda lemma.
\end{proof}

Morel's stable $\aone$-connectivity theorem used in the proof of Lemma \ref{lem:simplicialhurewicz} also equips the category $\Ab^{\aone}_k$ with a symmetric monoidal structure for which we will write $\tensor^{\aone}$.

\begin{defn}
\label{defn:aonetensorproduct}
If $\mathscr{M}$ and $\mathscr{M}'$ are strictly $\aone$-invariant sheaves, then
\[
\mathscr{M} \tensor^{\aone} \mathscr{M}' := \H_0^{\aone}(\mathscr{M} \tensor \mathscr{M}').
\]
\end{defn}

\subsubsection*{Contractions and $\gm$-loop spaces}
We begin by recalling the definition of contractions of a pointed sheaf, and then stating some properties that follow immediately from the definitions.

\begin{defn}
Suppose $\F$ is presheaf of pointed sets on $\Sm_k$.  The presheaf $\F_{-1}$ is the internal pointed function presheaf $\underline{\hom}_{\bullet}(\gm,\F)$, i.e., for any smooth scheme $U$ we have
\[
\F_{-1}(U) = \ker(\F(\gm \times U) \stackrel{ev_1}{\longrightarrow} \F(U)),
\]
where the map $ev_1$ is the pullback along the map $U \stackrel{id \times 1}{\to} U \times \gm$.
\end{defn}

\begin{lem}
Suppose $\F$ is a presheaf of pointed sets.  If $\F$ is a sheaf (resp. sheaf of groups, sheaf of abelian groups), then so is $\F_{-1}$.
\end{lem}

Suppose $\mathscr{E}$ is an $S^1$-spectrum and $\Sigma^{\infty}_s\gm$ is the symmetric spectrum associated with the space $\gm$ pointed by $1$.  We can consider the internal function spectrum $\underline{\hom}(\Sigma^{\infty}_s\gm,\mathscr{E})$.  The diagonal $U \to U \times U$ induces a map $U_+ \to U_+ \wedge U_+$ and consequently a morphism of spectra $\Sigma^{\infty}_s U_+ \to \Sigma^{\infty}_s U_+ \wedge \Sigma^{\infty}_s U_+$.  Any morphism $U \to \gm$ induces a morphism of pointed spaces $U_+ \to \gm$ and consequently a morphism of spectra $\Sigma^{\infty}_s U_+ \to \Sigma^{\infty}_s \gm$.  Combining these two observations, we get a map
\[
[\Sigma^{\infty}_s U_+ \wedge \Sigma^{\infty}_s \gm,\mathscr{E}]_s \times \hom_{\Sm_k}(U,\gm) \longrightarrow [\Sigma^{\infty}_s U_+ \wedge \Sigma^{\infty}_s U_+,\mathscr{E}]_s \longrightarrow [\Sigma^{\infty}_s U_+,\mathscr{E}]_s.
\]
Sheafifying for the Nisnevich topology and keeping track of basepoints, this corresponds to a map
\[
\bpi_0^{s}(\underline{\hom}(\Sigma^{\infty}_s \gm,\mathscr{E})) \wedge \gm \to \bpi_0^{s}(\mathscr{E}).
\]
By adjunction, such a morphism is equivalent to a morphism of sheaves
\[
\bpi_0^{s}(\underline{\hom}(\Sigma^{\infty}_s \gm,\mathscr{E})) \to \bpi_0^{s}(\mathscr{E})_{-1}.
\]
Similarly, for any integer $n \in \Z$ there is an induced morphism
\begin{equation}
\label{eqn:homotopygroupscontraction}
\bpi_i^{s}(\underline{\hom}(\Sigma^{\infty}_s \gm,\mathscr{E})) \to \bpi_i^{s}(\mathscr{E})_{-1}.
\end{equation}
Assuming $\mathscr{E}$ is $\aone$-local produces a corresponding map of stable $\aone$-homotopy sheaves of $\mathscr{E}$.

\begin{prop}[{\cite[Lemma 4.3.11]{MIntro}}]
\label{prop:connectivityofgmloopspaces}
If $\mathscr{E}$ is an $\aone$-local $S^1$-spectrum, then for any $i \in \Z$ the morphism
\[
\bpi_i^{s}(\underline{\hom}(\Sigma^{\infty}_s \gm,\mathscr{E})) \to \bpi_i^{s}(\mathscr{E})_{-1}
\]
of \textup{Equation \ref{eqn:homotopygroupscontraction}} is an isomorphism.  In particular, if $\mathscr{E}$ is a $(-1)$-connected $\aone$-local spectrum, then for any integer $n \geq 0$, $\underline{\hom}(\Sigma^{\infty}_s \gm^{\wedge n},\mathscr{E})$ is also a $(-1)$-connected $\aone$-local spectrum.
\end{prop}

\begin{cor}
If $\mathscr{M}$ is a strictly $\aone$-invariant sheaf of groups, then so is $\mathscr{M}_{-1}$.
\end{cor}

\begin{proof}
The sheaf $\mathscr{M}$ is a strictly $\aone$-invariant sheaf if and only if the Eilenberg-MacLane spectrum $\operatorname{H}\mathscr{M}[i]$ is $\aone$-local for all $i$.  Since $\operatorname{H}\mathscr{M}[i]$ is $\aone$-local, so is $\underline{\hom}(\Sigma^{\infty}_s\gm,\operatorname{H}\mathscr{M}[i])$.  The sheaves $\bpi_i^{s}(\mathscr{X})$ are strictly $\aone$-invariant if $\mathscr{X}$ is $\aone$-local.  The result then follows immediately from Proposition \ref{prop:connectivityofgmloopspaces}.
\end{proof}

\subsubsection*{Contractions in the $\aone$-derived setting}
We now rework the results of the previous section in the setting of the $\aone$-derived category.  Suppose $\mathscr{C}$ is an object of $\Sp^{\Sigma}(\Ch\Ab_k)$.  If $\Z \langle 1 \rangle$ is the enhanced Tate complex introduced before, we study the internal hom object $\underline{\hom}(\Z \langle 1 \rangle[1],\mathscr{C})$.  Mirroring the construction of the previous section in $\Daone^{\eff}(k)$, there is an induced morphism
\begin{equation}
\label{eqn:tatehomderived}
\H_i^{\aone}(\underline{\hom}(\Z \langle 1 \rangle[1],\mathscr{C})) \to \H_i^{\aone}(\mathscr{C})_{-1}.
\end{equation}
We now show that this morphism is an isomorphism.

\begin{prop}
\label{prop:aonederivedcontraction}
If $\mathscr{C}$ is an $\aone$-local object in $\Sp^{\Sigma}(\Ch\Ab_k)$, then the morphism of \textup{Formula \ref{eqn:tatehomderived}} is an isomorphism.
\end{prop}

\begin{proof}
The proof is formally identical to \cite[Lemma 4.3.11]{MIntro}.  One first proves the result in the case where $\mathscr{M}$ is a strictly $\aone$-invariant sheaf of groups.  In that case, it suffices to check the result on sections over fields.  We then just have to show that
\[
\hom_{\operatorname{D}(\Ab_k)}(\Z,\underline{\hom}(\Z\langle 1 \rangle[1],\mathscr{M})[i]) \longrightarrow \hom_{\operatorname{D}(\Ab_k)}(\Z,\mathscr{M}_{-1}[i])
\]
is an isomorphism.  By adjunction the map above can be rewritten as:
\[
\hom_{\operatorname{D}(\Ab_k)}(\Z\langle 1 \rangle[1],\mathscr{M}[i]) \longrightarrow \hom_{\operatorname{D}(\Ab_k)}(\Z,\mathscr{M}_{-1}[i]).
\]
However, the sheaf $\underline{\hom}(\Z\langle 1 \rangle [1])$ is precisely the $\aone$-chain complex of $\gm$, and since $\mathscr{M}$ is $\aone$-local the group on the left is an ordinary cohomology group.  To finish, we observe that $\tilde{C}_*^{\aone}(\pone) = \tilde{C}_*^{\aone}(\gm)[-1]$, and that both $\gm$ and $\pone$ have Nisnevich cohomological dimension $\leq 1$.  To treat the general case, we reduce to the one above by means of a Postnikov tower argument.
\end{proof}

\subsubsection*{The $\pone$-stable story}
We now prove the $\pone$-stable version of the Hurewicz theorem.

\begin{thm}
\label{thm:ponestablehurewiczisomorphism}
Suppose $\mathscr{X}$ is a $k$-space.  The Hurewicz morphism of \textup{Formula \ref{eqn:hurewiczhomomorphism}}
\[
\bpi_0^s(\Sigma^{\infty}_{\pone}\mathscr{X}_+) \longrightarrow \H_0^{s\aone}(\mathscr{X})
\]
is an isomorphism covariantly functorial in $\mathscr{X}$.  When $\mathscr{X} = \Spec k$, the aforementioned isomorphism reads $\K^{MW}_0 \isomto \H_0^{s\aone}(\1_k)$.  Moreover, via this identification, the Hurewicz isomorphism is compatible with the (left) action of $\K^{MW}_0$ on both the source and target in the sense that the following diagram commutes:
\[
\xymatrix{
\K^{MW}_0 \tensor^{\aone} \bpi_0^s(\Sigma^{\infty}_{\pone}\mathscr{X}_+) \ar[r]\ar[d] & \bpi_0^s(\Sigma^{\infty}_{\pone}\mathscr{X}_+) \ar[d] \\
\K^{MW}_0 \tensor^{\aone} \H_0^{s\aone}(\mathscr{X}) \ar[r] & \H_0^{s\aone}(\mathscr{X}).
}
\]
\end{thm}

\begin{proof}
The last statement regarding module structures follows from the Hurewicz isomorphism via Proposition \ref{prop:hurewiczcommuteswithmodulestructures} and Definition \ref{defn:stableaonehomotopygroupsofspheres}.  Thus, it suffices to prove that the Hurewicz homomorphism is an isomorphism.

By Proposition \ref{prop:stablehomotopycolimit}, $\bpi_0^s(\Sigma^{\infty}_{\pone}\mathscr{X}_+)$ is the sheaf associated with the presheaf
\[
U \mapsto \colim_n \hom_{\SH^{S^1}(k)}(\Sigma^{\infty}_s\gm^{\wedge n} \wedge \Sigma^{\infty}_s U_+,\Sigma^{\infty}_s \gm^{\wedge n} \wedge \Sigma^{\infty}_s \mathscr{X}_+).
\]
Equivalently, by adjunction, we can identify $\bpi_0^s(\Sigma^{\infty}_{\pone}\mathscr{X}_+)$ with the sheaf
\[
\colim_n \bpi_0^{s}(\mathbf{R}\hom_{\bullet}(\Sigma^{\infty}_s\gm^{\wedge n},\Sigma^{\infty}_s \gm^{\wedge n} \wedge \Sigma^{\infty}_s \mathscr{X}_+)).
\]
Proposition \ref{prop:connectivityofgmloopspaces} gives an isomorphism
\[
\bpi_0^{s}(\mathbf{R}\hom_{\bullet}(\Sigma^{\infty}_s\gm^{\wedge n},\Sigma^{\infty}_s \gm^{\wedge n} \wedge \Sigma^{\infty}_s \mathscr{X}_+)) \isomto \bpi_0^{s}(\Sigma^{\infty}_s \gm^{\wedge n} \wedge \Sigma^{\infty}_s \mathscr{X}_+))_{-n}.
\]

Lemma \ref{lem:simplicialhurewicz} shows that the map
\[
\bpi_0^{s}(\Sigma^{\infty}_s \gm^{\wedge n} \wedge \Sigma^{\infty}_s \mathscr{X}_+) \to \tilde{\H}_0^{\aone}(\gm^{\wedge n} \wedge \mathscr{X}_+)
\]
is an isomorphism and consequently that the induced map
\[
\bpi_0^{s}(\Sigma^{\infty}_s \gm^{\wedge n} \wedge \Sigma^{\infty}_s \mathscr{X}_+)_{-n} \to \tilde{\H}_0^{\aone}(\gm^{\wedge n} \wedge \mathscr{X}_+)_{-n}
\]
is an isomorphism.

On the other hand, Proposition \ref{prop:stablederivedcolimit} together with an adjunction argument show that $\H_0^{s\aone}(\mathscr{X})$ can be written as
\[
\colim_n \H_0^{\aone}(\underline{\hom}(\Z \langle n\rangle [n],C_*^{\aone}(\mathscr{X})\langle n \rangle[n])).
\]
Proposition \ref{prop:aonederivedcontraction} then identifies this group with
\[
\colim_n \H_0^{\aone}(C_*^{\aone}(\mathscr{X})\langle n \rangle[n])_{-n}.
\]
However, the groups in the colimit are precisely $\tilde{\H}_0^{\aone}(\mathscr{X}_+ \wedge \gm^{\wedge n})_{-n}$.  Since the simplicial Hurewicz homomorphism is compatible with the tensor structures the result follows.
\end{proof}

\section{Duality and $\aone$-homology}
\label{s:duality}
Suppose $k$ is a field.  Section \ref{ss:thomspaces} studies Thom spaces and their cohomology and connectivity properties.  In Section \ref{ss:enhancedmotiviccomplexes}, we introduce complexes $\Zn$ in the $\aone$-derived category (see Definition \ref{defn:enhancedmotiviccomplex}) that are formally very similar to the usual motivic complexes (see, e.g., \cite[\S 3]{MVW}).  We then investigate a cohomology theory associated with these complexes, analogous to motivic cohomology, paying special attention to the case of Thom spaces of vector bundles.  We introduce two versions of this theory in the unstable and stable settings (see Definition \ref{defn:enhancedmotiviccohomology}).

By constraining the degree and dimensions of the spaces in question, we show in Section \ref{ss:stablerepresentability} that sometimes the unstable groups coincide with the stable groups (see Corollary \ref{cor:stablerepresentability}).  In order to do this, we need to recall facts about homotopy modules (see Definition \ref{defn:homotopymodule}), which are objects of the heart of the so-called homotopy $t$-structure on $\Daone(k)$.  In addition to their use in proving the aforementioned stable representability result, homotopy modules serve a second purpose: using formal properties of homotopy modules and the homotopy $t$-structure, we show in Section \ref{ss:furtherreductions} that the zeroth $\aone$-homology sheaf is a birational invariant of smooth proper varieties; thus, it suffices to describe this sheaf for smooth projective varieties by Chow's lemma.

Finally, having reduced to the study of $\aone$-homology of smooth projective varieties, in Section \ref{ss:atiyahduality} we study the analog of Atiyah duality in the stable $\aone$-derived category (see Proposition \ref{prop:atiyahduality}).  Combining all of these results, we provide the analogs of steps (a) - (h) of the outline given in the proof of Lemma \ref{lem:suslinchow} in the context of stable $\aone$-homology sheaves.  The key issue we have to consider is that duality is only defined in the stable $\aone$-derived category, while essentially all of the other constructions and computations we make are ``unstable."  The main result of the section is Theorem \ref{thm:aonehomologyintermsofthomspaces}, which provides the first part of our description of the stable $\aone$-homotopy sheaves of groups.

\subsection{Properties of Thom spaces}
\label{ss:thomspaces}
If $\xi: E \to X$ is a vector bundle with zero section $i: X \to E$, recall that $Th(\xi) = E/E - i(X)$.  If $\xi': E' \to X$ is another vector bundle, and $f: E \to E'$ is a morphism of vector bundles that is injective on fibers, then there is an induced morphism $Th(\xi) \to Th(\xi')$.  We use the following results about Thom spaces repeatedly in the sequel.

\begin{prop}[{\cite[\S 3 Proposition 2.17]{MV}}]
\label{prop:basicpropertiesofthomspacesI}
Suppose $X_1$ and $X_2$ are smooth $k$-schemes, and $\xi_1: E_1 \to X_1$ and $\xi_2: E_2 \to X_2$ are vector bundles.
\begin{itemize}
\item[i)] There is a canonical isomorphism of pointed spaces $Th(\xi_1 \times \xi_2) \isomt Th(\xi_1) \wedge Th(\xi_2)$.
\item[ii)] If $\xi_1$ is a trivial rank $n$ vector bundle, there is an $\aone$-weak equivalence
\[
Th(\xi_1) \isomto {\pone}^{\wedge n} \wedge {X_1}_+.
\]
\end{itemize}
\end{prop}

\begin{prop}
\label{prop:basicpropertiesofthomspacesII}
Suppose $X$ and $X'$ are smooth schemes, $Z \subset X$ is a smooth closed subscheme, $f: X' \to X$ is a morphism, and $Z' = f^{-1}(Z)$.  Let $\nu_{Z'/X'}: N_{Z'/X'} \to Z'$ and $\nu_{Z/X}: N_{Z/X} \to Z$ be the associated normal bundles. If the induced map $N_{Z'/X'} \to f^*N_{Z/X}$ is an isomorphism, then the diagram
\[
\xymatrix{
X'/(X'-Z') \ar[r]\ar[d] & X/X-Z \ar[d] \\
Th(\nu_{Z'/X'}) \ar[r] & Th(\nu_{Z/X}),
}
\]
where the vertical maps are the purity isomorphisms and the horizontal maps are induced by $f$, commutes.
\end{prop}

\subsubsection*{Thom spaces in short exact sequences}
\begin{prop}
\label{prop:thomspaceadditivityinexactsequences}
Suppose $\xi: E \to X$, $\xi': E' \to X$ and $\xi'': E'' \to X$ are three vector bundles over $X$ and assume there is a short exact sequence of vector bundles
\[
0 \longrightarrow E' \longrightarrow E \longrightarrow E'' \longrightarrow 0.
\]
There is an isomorphism in the $\aone$-homotopy category
\[
Th(\xi) \isomto Th(\xi') \wedge Th(\xi'')
\]
that is unique up to unique isomorphism.
\end{prop}

\begin{proof}
Since $X$ is a smooth variety we can choose a smooth affine vector bundle torsor $\pi: \tilde{X} \to X$.  Pulling back the exact sequence from $X$ to $\tilde{X}$, we get a short exact sequence
\[
0 \longrightarrow \tilde{E'} \longrightarrow \tilde{E} \longrightarrow \tilde{E''} \longrightarrow 0,
\]
where $\tilde{\xi}: \tilde{E} \to \tilde{X}$ (resp. $\tilde{\xi}',\tilde{E}'$, $\tilde{\xi}'',\tilde{E}''$) are the pullbacks of $E$ (resp. $E'$, $E''$) to $\tilde{X}$.  Now, since $\tilde{X}$ is affine, this short exact sequence of vector bundles splits.  Choice of a splitting gives an isomorphism $\tilde{E} \isomto \tilde{E}' \oplus \tilde{E}''$, and by Proposition \ref{prop:basicpropertiesofthomspacesI}(i), we get an isomorphism $Th(\tilde{\xi}) \isomt Th(\tilde{\xi}') \wedge Th(\tilde{\xi''})$.

Since the morphism $\pi$ is an $\aone$-weak equivalence, and pushouts of $\aone$-weak equivalences along cofibrations are again $\aone$-weak equivalences (the $\aone$-local model structure is proper), it follows that the induced morphisms $Th(\tilde{\xi}) \to Th(\xi)$ (resp. $Th(\tilde{\xi'}) \to Th(\xi')$, $Th(\tilde{\xi}'') \to Th(\xi'')$) are again $\aone$-weak equivalences.  Thus there is an  isomorphism in $\hop{k}$
\[
Th(\xi) \longrightarrow Th(\xi') \wedge Th(\xi'').
\]
One can check that the space of splittings is an inductive limit of linear spaces and thus $\aone$-contractible.  It follows that the induced map in the $\aone$-homotopy category is unique up to unique isomorphism.

Finally, we show that the construction above does not depend on the choice of affine vector bundle torsor $\pi$.  If $\pi': \tilde{X}' \to X$ is another affine vector bundle torsor, then the fiber product $\tilde{X}'' := \tilde{X}' \times_X \tilde{X} \to X$ is also an affine vector bundle torsor that maps to both $\tilde{X}$ and $\tilde{X'}$ by projections.  Moreover, these projections make $\tilde{X}''$ into a vector bundle over $\tilde{X}$ or $\tilde{X'}$ since affine vector bundle torsors over affine varieties are simply vector bundles.  Running the construction above on each of these spaces (i.e., pulling back in sequence to $\tilde{X}$, $\tilde{X'}$ and $\tilde{X}''$) gives the necessary independence.
\end{proof}

\subsubsection*{Thom spaces of virtual bundles}
If $\xi: E \to X$ is a vector bundle, we can consider the suspension spectrum $\Sigma^{\infty}_{\pone}(Th(\xi))$, which we refer to as the Thom spectrum of $\xi$.  Given a virtual vector bundle on a smooth variety $X$, that is, a formal difference $\xi = \xi_1 - \xi_2$ of vector bundles, we can construct an associated Thom spectrum as follows.

First, we choose an affine vector bundle torsor $\tilde{X} \to X$; let $\tilde{\xi}_1$ and $\tilde{\xi_2}$ be the pullback vector bundles.  Since $\tilde{X}$ is affine, we can find a vector bundle $\xi_3$ over $\tilde{X}$ such that $\tilde{\xi_2}\oplus \xi_3 \isomto O_X^{\oplus n}$ for some $n$. Now we define the Thom spectrum $\Sigma^{\infty}_{\pone}(Th(\xi))$ (note this is an abuse of notation as this spectrum is not in fact a suspension spectrum) by
\[
\Sigma^{\infty}_{\pone}(Th(\xi)) = \Sigma^{-n}_{\pone}(\Sigma^{\infty}_{\pone}(Th(\xi_1))\wedge \Sigma^{\infty}_{\pone}(Th(\xi_3))).
\]
By Proposition \ref{prop:thomspaceadditivityinexactsequences}, the Thom spectrum is well-defined up to unique isomorphism in the stable homotopy category $\SH(k)$, and its isomorphism class only depends on the class of $\xi$ in $K_0(X)$. In an analogous fashion, we can construct Thom objects $\tilde{C}_*^{\aone}(Th(\xi))$ in ${\Daone(k)}$ for any virtual bundle $\xi$, well defined up to unique isomorphism.

The following example indicates the way in which this discussion will be used.

\begin{ex}
\label{ex:thomspaceofavirtualbundle}
Suppose $X$ is a smooth scheme, and $\xi: E \to X$ and $\xi': E' \to X$ are vector bundles on $X$.  Suppose we have an equation $[E \oplus O_X^{\oplus n}] = [E']$ in $K_0(X)$.  Applying Proposition \ref{prop:basicpropertiesofthomspacesI}(i) and (ii), we get an identification $\Sigma^{\infty}_{\pone}Th(\xi) \wedge {\pone}^{\wedge n} \isomt \Sigma^{\infty}_{\pone}Th(\xi')$, or equivalently an identification $\Sigma^{\infty}_{\pone}(Th(\xi)) \isomt \Sigma_{\pone}^{-n} \Sigma^{\infty}_{\pone}Th(\xi')$.
\end{ex}

\subsubsection*{Connectivity and cohomological dimension of Thom spaces}
\begin{prop}
\label{prop:thomspaceconnectivity}
If $X$ is a smooth $k$-scheme, and $\xi: E \to X$ is a rank $n$ vector bundle over $X$, then $Th(\xi)$ is $(n-1)$-$\aone$-connected.
\end{prop}

\begin{proof}
If $\xi: E \to X$ is a trivial bundle, then we have the identification $Th(\xi) \isomto X_+ \wedge (\pone)^{\wedge n}$; the result then follows immediately from the unstable $\aone$-connectivity theorem \cite[\S 2 Corollary 3.22]{MV}.  If $U \to V$ is an open immersion, then there is an obvious morphism of vector bundles $\xi|_U \to \xi$.  This induces a map $Th(\xi|_U) \to Th(\xi)$.  Functoriality of pushouts thus gives rise to a morphism from the pushout of $Th(\xi|_U)$ and $Th(\xi|_V)$ along $Th(\xi|_{U \cap V})$ to $Th(\xi)$; this morphism is an isomorphism.

Thus, suppose we have an open cover $U,V$ of $X$ and assume we know that $Th(\xi|_V)$, $Th(\xi|_U)$ and $Th(\xi|_{U \cap V})$ are all $n$-connected.  We then have a pushout diagram
\[
\xymatrix{
Th(\xi|_{U \cap V}) \ar[r]\ar[d] & Th(\xi|_V) \ar[d] \\
Th(\xi|_{U}) \ar[r] & Th(\xi).
}
\]
If $n = 1$, the result is an immediate consequence of the unstable $\aone$-connectivity theorem \cite[Theorem 5.37]{MField}.  If $n = 2$, we choose base-points and the result follows by an argument using the $\aone$-van Kampen theorem \cite[Theorem 6.12]{MField}.  If $n > 2$, we use the (unstable) $\aone$-Hurewicz theorem \cite[Theorem 5.36]{MField} to deduce the result by induction.  Indeed, if $\mathscr{X}$ is an $i$-connected space ($i \geq 1$), then the induced map from $\bpi_{i+1}^{\aone}(\mathscr{X},x) \to \H_{i+1}^{\aone}(\mathscr{X},x)$ is an isomorphism for $i > 1$.  In that case, the result follows from a Mayer-Vietoris argument.
\end{proof}

\begin{rem}
Assume hypotheses are as in the statement of Proposition \ref{prop:thomspaceconnectivity}.  If $\mathscr{M}$ is a strictly $\aone$-invariant sheaf, one can show that $H^i_{\Nis}(Th(\xi),\mathscr{M})$ vanishes for $i \leq n-1$ in a ``purely stable" manner.  By the stable $\aone$-connectivity theorem, one knows that the Thom spectrum of a trivial bundle is stably $(n-1)$-connected.  Mayer-Vietoris can be used to conclude the argument in general.  In the sequel, we will only ever use this weaker vanishing statement.  In a different direction, another consequence of the unstable Hurewicz theorem is that the canonical morphism $\bpi_n^{\aone}(Th(\xi)) \to \H_n^{\aone}(Th(\xi))$ is an isomorphism.
\end{rem}

Recall that for any $\mathscr{X}$ an object of $\Spc_k$ and any complex of Nisnevich sheaves of abelian groups $F^\bullet$ on $\Sm_k$ we can define the hypercohomology groups $\mathbb{H}_{\Nis}^p (\mathscr{X}, F^\bullet)$; we say that a space $\mathscr{X}$ has cohomological dimension $\leq d$ provided that for every sheaf $F$ (viewed as complex concentrated in degree $0$), the cohomology groups $H_{\Nis}^p(\mathscr{X},F) = 0$ for all $p > d$.

\begin{lem}
\label{lem:cohomologicaldimensionofthomspaces}
If $\xi: E \to X$ is a rank $r$ vector bundle over a smooth $k$-scheme $X$, then
\[
cd_{\Nis}(Th(\xi)) <= \dim X + r.
\]
\end{lem}

\begin{proof}
If $i: X \to E$ denotes the zero section of $\xi$, then by definition of $Th(\xi)$, we have a cofibration sequence
\[
E - i(X) \longrightarrow E \longrightarrow Th(\xi) \longrightarrow \Sigma^1_s (E - i(X))_+ \longrightarrow \cdots.
\]
This cofibration sequence induces a long exact sequence in cohomology with coefficients in an arbitrary abelian sheaf.  The result then follows from the usual Nisnevich cohomological dimension theorem, i.e., that the Nisnevich cohomological dimension of a scheme is bounded above by its Krull dimension.
\end{proof}

\subsection{Enhanced motivic complexes and cohomology}
\label{ss:enhancedmotiviccomplexes}
If $k$ is a perfect field, motivic cohomology of a smooth $k$-scheme $X$ can be defined in terms of the motivic complexes $\Z(n)$ and the motive of $X$ by the formula (see, e.g., \cite[Definition 14.7]{MVW})
\[
H^{i,n}(X,\Z) := \hom_{\DM^{\eff}(k)}(\M(X),\Z(n)[i]).
\]
There is a canonical map
\[
\hom_{\DM^{\eff}(k)}(\M(X),\Z(n)[i]) \to \hom_{\DM(k)}(\M(X),\Z(n)[i])
\]
comparing the ``unstable" motivic cohomology groups with their ``stable" counterparts.  In the motivic context, this map is known to be an isomorphism for arbitrary pairs $(i,n)$ by Voevodsky's cancelation theorem \cite[Corollary 4.10]{VCancellation}.

Our goal in this section is to define an analog of motivic cohomology where $\DM^{\eff}(k)$ (resp. $\DM(k)$) is replaced by $\Daone^{\eff}(k)$ (resp. $\Daone(k)$).  The analog of $\M(X)$ in this context is the $\aone$-chain complex $C_*^{\aone}(X)$.  Producing a theory that has a hope of being stably representable will involve some contortions.

\subsubsection*{Enhanced motivic complexes}
We now define enhanced motivic complexes; these complexes play the role of the motivic complexes just discussed in the $\aone$-derived category.  The resulting complexes are mapped to the motivic complexes by the derived functor of ``adding transfers" from the $\aone$-derived category to Voevodsky's derived category of effective motivic complexes.  

\begin{defn}
\label{defn:enhancedmotiviccomplex}
The {\em enhanced motivic complex of weight $n$} is defined by the formula
\[
\Zn := \tilde{C}_*^{\aone}({\pone}^{\wedge n})[-2n].
\]
\end{defn}

\begin{lem}
The map $\Z\langle 1 \rangle^{\tensor n} \to \Zn$ induced by the morphism of spaces ${\pone}^{\times n} \to {\pone}^{\wedge n}$ is an isomorphism.
\end{lem}

\begin{defn}
\label{defn:enhancedmotiviccohomology}
If $\mathscr{X}$ is an object of $\Spc_k$, the {\em unstable $(i,n)$ enhanced motivic cohomology of $\mathscr{X}$} is the group $\hom_{\Daone^{\eff}(k)}(C_*^{\aone}({\mathscr X}),\Zn [i])$.  The {\em stable $(i,n)$ enhanced motivic cohomology of ${\mathscr X}$} is the group $\hom_{\Daone(k)}(C_*^{s\aone}({\mathscr X}),\Zn [i])$.
\end{defn}

\begin{rem}
\label{rem:naivecancellationfails}
Observe that according to Definition \ref{defn:enhancedmotiviccomplex}, the convention that ${\pone}^{\wedge 0} = S^0_s$ implies that $\Z\langle 0 \rangle := \tilde{C}_*^{\aone}(S^0_s) = \Z$.  One consequence of this observation is that the unstable group $\hom_{\Daone^{\eff}(k)}(C_*^{\aone}(\Spec k),\Z\langle 0 \rangle) = \Z$ does {\em not} coincide with its stable counterpart $\hom_{\Daone(k)}(C_*^{s\aone}(\Spec k),\Z\langle 0 \rangle)$.  Indeed, the latter sheaf has been computed by Morel to be $\K^{MW}_0$.
\end{rem}

\subsubsection*{Basic properties of enhanced motivic cohomology}
The following result summarizes some properties of the enhanced motivic cohomology groups and the complexes $\Zn$.

\begin{prop}
\label{prop:enhancedmotiviccomplexes}
Suppose ${\mathscr X}$ is a space.
\begin{itemize}
\item[i)] For any integers $i,n$, there is a canonical identification
\[
{\mathbb H}^{i}_{\Nis}({\mathscr X},\Z \langle n \rangle) \isomto \hom_{\Daone^{\eff}(k)}(C_*^{\aone}({\mathscr X}),\Z\langle n \rangle [i]).
\]
\item[ii)] For $i > n$, the cohomology sheaves $\underline{H}^i(\Z \langle n \rangle)$ vanish.
\item[iii)] For any integer $n > 0$, there is a canonical isomorphism $\underline{H}^n(\Z \langle n \rangle) = \K^{MW}_n$.
\end{itemize}
\end{prop}

\begin{proof}
Statement (i) follows immediately from the observation that the complex $\Z \langle n \rangle$ is $\aone$-local, which is implicit in its definition.  Statements (ii) and (iii) are equivalent to the assertions that (ii)': $\tilde{\H}_i^{\aone}(\gm^{\wedge n}) = 0$ for $i < 0$ and (iii'): $\tilde{\H}_0^{\aone}(\gm^{\wedge n}) = \K^{MW}_n$ after trading cohomological indexing for its homological counterpart.  Statement (ii)' is a consequence of the stable $\aone$-connectivity theorem \cite[Theorem 6.1.8]{MStable} since the complex $\Z(\gm^{\wedge n})$ is $(-1)$-connected.  Statement (iii)' is a consequence of the stable Hurewicz theorem \ref{thm:ponestablehurewiczisomorphism} and Definition \ref{defn:stableaonehomotopygroupsofspheres}.
\end{proof}


Statements (ii) and (iii) of Proposition \ref{prop:enhancedmotiviccomplexes} imply that the hypercohomology spectral sequence gives rise to a well-defined morphism
\begin{equation}
\label{eqn:edgemap}
{\mathbb H}_{\Nis}^{p}({\mathscr X},\Zn) \longrightarrow H^{p-n}_{\Nis}({\mathscr X},\K^{MW}_n).
\end{equation}
Since the construction of the hypercohomology spectral sequence is functorial in ${\mathscr X}$, it follows that the morphism above is actually natural in ${\mathscr X}$.  The next result gives a condition under which this morphism is an isomorphism.

\begin{cor}
\label{cor:enhancedmotivicischowwitt}
Suppose ${\mathscr X}$ is a space having Nisnevich cohomological dimension $\leq d$.  The map of \textup{Equation \ref{eqn:edgemap}} is an isomorphism for $p \geq n + d$.
\end{cor}

\begin{proof}
The kernel and cokernel of the map in Equation \ref{eqn:edgemap} are built out of groups of the form $H^{p-i}_{\Nis}({\mathscr X},\underline{H}^i(\Zn))$ and $H^{p-i+1}_{\Nis}({\mathscr X},\underline{H}^i(\Zn))$ with $i < n$.  By assumption $p - n \geq d$, so if $i < n$, then $p-i > p-n \geq d$.  Since the Nisnevich cohomological dimension of ${\mathscr X}$ is $\leq d$, it follows both of the aforementioned groups are trivial.
\end{proof}

\subsubsection*{Comparison with motivic cohomology}
Assume $k$ is a perfect field.  For every integer $n \geq 0$, the functor $\Daone^{\eff}(k) \to \DM^{\eff}(k)$ sends the complexes $\Zn$ to $\Z(n)$ by construction.  View the complex $\Z(n)$ as an object in $\Daone^{\eff}(k)$ via the functor of ``forgetting transfers."

\begin{lem}
\label{lem:inducedbyh_nsheaves}
Assume $k$ is a perfect field.  Suppose $\mathscr{X}$ is a space of cohomological dimension $d > 0$.  For every pair of integers $(p,n)$ the following diagram commutes
\[
\xymatrix{
{\mathbb H}_{\Nis}^{p}({\mathscr X},\Zn) \ar[r]\ar[d] & H^{p-n}_{\Nis}({\mathscr X},\K^{MW}_n) \ar[d] \\
{\mathbb H}_{\Nis}^{p}({\mathscr X},\Z(n)) \ar[r] & H^{p-n}_{\Nis}({\mathscr X},\K^M_n).
}
\]
where the left vertical map is induced by the morphism of complexes $\Zn \to \Z(n)$, the horizontal maps are the induced maps in the hypercohomology spectral sequences and the right vertical map is the morphism of sheaves $\K^{MW}_n \to \K^M_n$ induced by applying the functor $n$-th cohomology sheaf to the morphism of complexes $\Zn \to \Z(n)$.
\end{lem}

\begin{proof}
This follows immediately from functoriality of the hypercohomology spectral sequences: observe that the induced maps on cohomology sheaves of degree $> n$ are isomorphisms since these sheaves vanish for both complexes.

\end{proof}

If ${\mathscr X}$ has Nisnevich cohomological dimension $\leq d$, and $p \geq n + d$, then both horizontal maps are isomorphisms.  As a consequence, in this range of degrees the left vertical map coincides with the map on cohomology induced by the morphism of sheaves $\K^{MW}_n \to \K^M_n$.

\subsubsection*{Enhanced motivic cohomology of Thom spaces}
\begin{cor}
\label{cor:enhancedmotiviccohomologyofthomspaces}
If $\xi: E \to X$ is a rank $r$ vector bundle over a smooth $k$-scheme $X$ of dimension $d$, then the canonical morphism
\[
\hom_{\Daone^{\eff}(k)}(C_*^{\aone}(Th(\xi)),\Z \langle d+r \rangle[2(d+r)]) \longrightarrow H^{d+r}_{\Nis}(Th(\xi),\K^{MW}_{d+r})
\]
of \textup{Equation \ref{eqn:edgemap}} is an isomorphism.
\end{cor}

\begin{proof}
Combine Corollary \ref{cor:enhancedmotivicischowwitt} and Lemma \ref{lem:cohomologicaldimensionofthomspaces}.
\end{proof}

\subsection{Stable representability in top codimension}
\label{ss:stablerepresentability}
A naive analog of Voevodsky's cancelation theorem in the setting of the $\aone$-derived category is false, as observed in Remark \ref{rem:naivecancellationfails}.  Thus, one must be somewhat careful in attempting to formulate any kind of stable representability result.

Nevertheless, the goal of this section is to show that certain unstable enhanced motivic cohomology groups can be identified with their stable counterparts.  Our result amounts to a very weak form of Voevodsky's cancelation theorem.  To prove the result, we need to recall various facts about the interaction between $\gm$-looping and Nisnevich cohomology.

\subsubsection*{Homotopy modules}
The category $\Daone(k)$ admits a $t$-structure whose heart we now describe.

\begin{defn}
\label{defn:homotopymodule}
A {\em homotopy module} over $k$ is a pair $(\mathscr{M}_{*},\varphi_*)$ consisting of a $\Z$-graded strictly $\aone$-invariant sheaf $\mathscr{M}_*$ and, for each $n \in \Z$, an isomorphism of abelian sheaves
\[
\varphi_n:  \mathscr{M}_n \isomto (\mathscr{M}_{n+1})_{-1}.
\]
We write $\Ab^{s\aone}_k$ for the category of homotopy modules.
\end{defn}

\begin{lem}
\label{lem:homotopymodule}
If $\mathscr{C}$ is an object of $\Daone(k)$, then for any integer $i$, the sheaves $\H_i^{\aone}(\mathscr{C} \langle n \rangle[n])$ form a homotopy module with structure maps given by the isomorphisms of \textup{Proposition \ref{prop:aonederivedcontraction}}.
\end{lem}

Proposition \ref{prop:aonederivedcontraction} shows that if $\mathscr{C}$ is a $(-1)$-connected $\aone$-local complex then all of the complexes $\mathscr{C}\langle n \rangle[n]$ are also $(-1)$-connected.  One can use this observation to define a $t$-structure on the category $\Daone(k)$: the ``positive" objects are those for which $\H_i^{\aone}(\mathscr{C}\langle n \rangle[n])$ vanishes for $i < 0$ and for all $n  \in \Z$.  Similarly, the ``negative" objects are those for which $\H_i^{\aone}(\mathscr{C}\langle n \rangle[n])$ vanishes for $i > 0$ and for all $n  \in \Z$.  The following result is an analog of \cite[Theorem 5.2.6]{MIntro} in the stable $\aone$-derived category; the proof follows from this result via Theorem \ref{thm:ponestablehurewiczisomorphism}.

\begin{thm}
\label{thm:aonederivedheart}
The functor $\mathscr{C} \mapsto \bigoplus_{n \in N} \H_0^{\aone}(\mathscr{C}\langle n \rangle [n])$ identifies the heart of the homotopy $t$-structure on $\Daone(k)$ with the category of homotopy modules.
\end{thm}

\subsubsection*{Contractions of Milnor-Witt K-theory sheaves}
\begin{prop}
\label{prop:contractionsofmilnorwittsheaves}
We have a canonical identification $(\K^{MW}_n)_{-1} \isomto \K^{MW}_{n-1}$.
\end{prop}

\begin{proof}
This is a consequence of Proposition \ref{prop:aonederivedcontraction}.
\end{proof}

\subsubsection*{Stable representability}
Suppose $({\mathscr X},x)$ is a pointed $k$-space.  We have a natural map
\[
\hom_{\Daone^{\eff}(k)}(C_*^{\aone}({\mathscr X}),\Z\langle n \rangle [2n]) \longrightarrow \hom_{\Daone^{\eff}(k)}(C_*^{\aone}({\mathscr X}) \tensor \Z\langle 1 \rangle [2],\Z\langle n \rangle [2n] \tensor \Z\langle 1 \rangle [2]).
\]
We have identifications $\Z\langle n \rangle [2n] \tensor \Z\langle 1 \rangle [2] \isomt \Z \langle n+1 \rangle[2(n+1)]$ and $C_*^{\aone}(X_+ \wedge \pone) \isomt C_*^{\aone}(X) \tensor \Z\langle 1 \rangle [2]$.  With these identifications, we have a map
\[
{\mathbb H}_{\Nis}^{2n}({\mathscr X},\Z\langle n \rangle) \longrightarrow {\mathbb H}^{2(n+1)}_{\Nis}({\mathscr X} \wedge \pone,\Z \langle n+1 \rangle).
\]
The group ${\mathbb H}^{2(n+1)}_{\Nis}({\mathscr X} \wedge \pone,\Z \langle n+1 \rangle)$ is a summand of ${\mathbb H}^{2(n+1)}_{\Nis}({\mathscr X} \times \pone,\Z \langle n+1 \rangle)$.

\begin{prop}
\label{prop:gmsuspensionofmilnorwittgroups}
For any pointed space $({\mathscr X},x)$ there is a canonical isomorphism
\[
H^n_{\Nis}({\mathscr X},\K^{MW}_n) \longrightarrow H^{n+1}_{\Nis}({\mathscr X} \wedge \pone,\K^{MW}_{n+1}).
\]
\end{prop}

\begin{proof}
This is a consequence of the fact that the sheaves $\K^{MW}_n$ fit together to form a homotopy module (which is, in turn, a consequence of Proposition \ref{prop:contractionsofmilnorwittsheaves}).  In fact, the result holds for any such homotopy module. Nevertheless, here are the details.  For any pointed space $\mathscr{X}$ we have identifications
\[
H^n_{\Nis}({\mathscr X},\K^{MW}_n) = [{\mathscr X},K(\K^{MW}_n,n)]_s \isomto [{\mathscr X},K(\K^{MW}_n,n)]_{\aone}.
\]
Indeed, the first equality is \cite[\S 2 Proposition 1.26]{MV}, and the second identification follows from the fact that the sheaf $\K^{MW}_n$ is strictly $\aone$-invariant, which implies that $K(\K^{MW}_n,n)$ is $\aone$-local.  Now, we can replace ${\mathscr X} \wedge \pone$ by $S^1_s \wedge \gm \wedge \mathscr{X}$ up to $\aone$-weak equivalence.  In the $S^1$-stable homotopy category the suspension isomorphism has been inverted and by adjunction we get identifications
\[
[S^1_s \wedge \gm \wedge \mathscr{X},K(\K^{MW}_{n+1},n+1)]_{\aone} \isomto [\gm \wedge \mathscr{X},K(\K^{MW}_{n+1},n)]_{\aone} \isomto [\mathscr{X},\Omega_{\gm}K(\K^{MW}_{n+1},n)].
\]

Now, there is an obvious map $\gm \wedge \K^{MW}_{n} \to \K^{MW}_{n+1}$ coming from the definition of the sheaves $\K^{MW}_n$ as free strictly $\aone$-invariant sheaves generated by the sheaves of sets $\gm^{\wedge n}$.  For an arbitrary smooth scheme $U$, we have
\[
H^n_{\Nis}(U \wedge \gm,\K^{MW}_{n+1}) \isomto H^n_{\Nis}(U,(\K^{MW}_{n+1})_{-1}) = H^n_{\Nis}(U,\K^{MW}_n)
\]
by means of Proposition \ref{prop:contractionsofmilnorwittsheaves}.  Under these identifications, the map of the statement is induced by the morphism of spaces $\K^{MW}_n \to \Omega_{\gm} \gm \wedge \K^{MW}_n$.  Thus, it suffices to compute $[{\mathscr X},\Omega_{\gm}K(\K^{MW}_{n+1},n)]$.  The result then follows immediately from \cite[Lemma 4.3.11 (p. 427)]{MIntro}.

\end{proof}

\begin{cor}
\label{cor:stablerepresentability}
For a space ${\mathscr X}$ of cohomological dimension $n > 0$, we have a canonical isomorphism
\[
\hom_{\Daone^{\eff}(k)}(C_*^{\aone}({\mathscr X}),\Z\langle n \rangle [2n]) \isomto \hom_{\Daone(k)}
(C_*^{\aone}({\mathscr X}),\Z\langle n \rangle [2n]).
\]
\end{cor}

\begin{proof}
The seemingly odd hypothesis on cohomological dimension is necessitated by Remark \ref{rem:naivecancellationfails}.  The isomorphism $\hom_{\Daone^{\eff}(k)}(C_*^{\aone}({\mathscr X}),\Z\langle n \rangle [2n]) \isomto H^n_{\Nis}({\mathscr X},\K^{MW}_n)$ obtained by combining Proposition \ref{prop:enhancedmotiviccomplexes}(i) and Corollary \ref{cor:enhancedmotivicischowwitt} is contravariantly functorial in ${\mathscr X}$.  Via this isomorphism, the result follows immediately from Proposition \ref{prop:gmsuspensionofmilnorwittgroups}.
\end{proof}

\subsection{Projectivity and Base change}
\label{ss:furtherreductions}
In this section, we prove that the sheaves $\H_0^{s\aone}(X)$ are birational invariants of smooth proper varieties if $k$ is an infinite field.  The sheaf $\H_0^{s\aone}(\mathscr{X})$ is by construction an object in the heart of the homotopy $t$-structure on $\Daone(k)$ (see Definition \ref{defn:homotopymodule} and Theorem \ref{thm:aonederivedheart}).  Moreover, almost by its construction, it is initial for such objects.  Precisely, we have the following result.

\begin{prop}
\label{prop:universality}
If $\mathscr{M}_*$ is a homotopy module, then there is a canonical bijection
\[
H^0_{\Nis}(\mathscr{X},\mathscr{M}_0) \isomto \hom_{{\Ab^{s\aone}_k}}(\H_0^{s\aone}(\mathscr{X})_*,\mathscr{M}_*)
\]
contravariantly functorial in $\mathscr{X}$.
\end{prop}

\begin{proof}
By definition $\mathscr{M}_*$ is an object lying in the heart of $\Daone(k)$.  Applying the functor $\H_0^{s\aone}$ induces a map
\[
H^0_{\Nis}(\mathscr{X},\mathscr{M}_0) \isomto \hom_{\Daone(k)}(C_*^{s\aone}(\mathscr{X}),\mathscr{M}_*) \longrightarrow \hom_{\Ab^{s\aone}_k}(\H_0^{s\aone}(\mathscr{X})_*,\mathscr{M}_*).
\]
A standard argument involving the axioms of a $t$-structure shows the second map is a bijection as well.
\end{proof}

\begin{prop}
\label{prop:birationalinvariance}
If $X$ and $X'$ are birationally equivalent smooth proper varieties, then $\H_0^{s\aone}(X) \cong \H_0^{s\aone}(X')$.
\end{prop}

\begin{proof}
Suppose $\mathscr{M}_*$ is a homotopy module.  By Proposition \ref{prop:universality} and the Yoneda lemma, it suffices to show that $\mathscr{M}_0(X) = \mathscr{M}_0(X')$.  However, the sheaf $\mathscr{M}_0$ is strictly $\aone$-invariant.  Since $k$ is infinite, the assertion follows immediately from \cite[Theorem 8.5.1]{CTHK}.
\end{proof}

\subsubsection*{Base Change}
The analog of the following result in the stable $\aone$-homotopy category of $S^1$-spectra is a consequence of \cite[Lemma 5.2.7]{MStable}.

\begin{prop}
\label{prop:basechange}
Suppose $\mathscr{X}$ is a $k$-space, and $L/k$ is a finitely generated separable extension.  Let $\mathscr{X}_L$ denote $\mathscr{X} \times_{\Spec k} \Spec L$.  We then have an identification
\[
\H_0^{\aone}(\mathscr{X})(L) \isomto \hom_{\Daone^{\eff}(L)}(\Z,C_*^{\aone}(\mathscr{X}_L)).
\]
\end{prop}

\begin{proof}
Suppose $X \to \Spec k$ is a smooth scheme such that $L = k(X)$.  Taking a directed system of neighborhoods of the generic point of $X$, we can write $L$ as an essentially smooth scheme, i.e., a filtering limit of smooth $k$-schemes with affine \'etale transition morphisms.  For any smooth morphism $f: S \to S'$, there is a pullback functor
\[
f^*: \Daone(S') \longrightarrow \Daone(S)
\]
Use the analog of \cite[Lemma 5.1.1]{MStable} to conclude.
\end{proof}

\subsection{Atiyah duality in the stable $\aone$-derived category}
\label{ss:atiyahduality}
The work of Cisinski-D\'eglise shows that one can apply the duality formalism to the category $\Daone(k)$.  In particular, the categories $\Daone(k)$ contain mapping objects.  Given a smooth variety $X$, we define $C_*^{s\aone}(X)^\vee$ to be the internal function object $\underline{\hom}(C_*^{s\aone}(X),\1)$.

\subsubsection*{Atiyah duality}
Our next goal is to provide a ``concrete" description of the dual.  What is now called Atiyah duality was introduced in \cite{Atiyah}, and has been studied in the context of stable $\aone$-homotopy theory by Hu \cite[Appendix A]{Hu} and Riou \cite{Riou}.  We begin by recalling the following result due to Voevodsky.

\begin{thm}[{\cite[Theorem 2.11]{VMilnor}}]
\label{thm:voevodskyduality}
Let $X$ be a smooth projective variety of pure dimension $d$ over a field $k$.  There exists an integer $n$ and a vector bundle $\nu: V \to X$ of rank $n$ such that
\begin{itemize}
\item[i)] if $\tau_{X}: T_X \to X$ is the tangent bundle to $X$, then $[V \oplus T_X] = [\O_X^{\oplus n +d}]$ in $K_0(X)$, and
\item[ii)] there exists a morphism $T^{n+d} \to Th(\nu)$ in $\hop{k}$ such that the induced map $H^{2d,d}(X) \to \Z$ coincides with the degree map.
\end{itemize}
\end{thm}

As a consequence of Theorem \ref{thm:voevodskyduality}(i) and Example \ref{ex:thomspaceofavirtualbundle}, we see that the Thom spectrum of the negative tangent bundle $-\tau_{X}$ coincides with the $(n+d)$-fold $\pone$-desuspension of the Thom spectrum of $\nu$.  As a consequence of Theorem \ref{thm:voevodskyduality}(ii), we deduce the existence of a morphism
\[
\1_k \to \Sigma^{\infty}_{\pone}Th(-\tau_{X}).
\]

\begin{prop}
\label{prop:atiyahduality}
If $X$ is a smooth projective variety, then $C_*^{s\aone}(X)$ is a strongly dualizable object of ${\Daone(k)}$ and its strong dual is $\tilde{C}_*^{s\aone}(Th(-\tau_X))$.
\end{prop}

\begin{proof}
By the definition of strong duality, we need to construct maps
\[
\eta: \1_k \to \tilde{C}_*^{s\aone}(Th(-\tau_{X})) \tensor C_*^{s\aone}(X)
\]
and
\[
\epsilon: \tilde{C}_*^{s\aone}(Th(-\tau_{X})) \tensor C_*^{s\aone}(X) \to \1_k.
\]
By definition of the tensor product, this is equivalent to specifying a map $\1_k \to \tilde{C}_*^{s\aone}(Th(-\tau_{X}) \wedge X_+)$.  We identify $X_+$ with the Thom space of the trivial rank $0$ bundle on $X$.  This map is then induced by Voevodsky's duality map.  Similarly, the map $\epsilon$ is induced by the Thom collapse map.  One then needs to check that a number of diagrams commute; this is accomplished in the same manner as \cite[Appendix A]{Hu}.
\end{proof}

For later convenience, we record the following consequence of Proposition \ref{prop:atiyahduality} and the hom-tensor adjunction.

\begin{lem}
\label{lem:existenceofduals}
If $X$ is a smooth projective variety we have a canonical and functorial isomorphism:
\[
\hom_{\Daone(k)}(\1_k,C_*^{s\aone}(X)) \isomto \hom_{\Daone(k)}(C_*^{s\aone}(X)^{\vee},\1_k).
\]
\end{lem}

\subsubsection*{Consequences of duality}
The next result is a variant of \cite[Lemma 2.1]{Atiyah} in the context of $\aone$-homology. We use the notation of Theorem \ref{thm:voevodskyduality}.

\begin{thm}
\label{thm:aonehomologyintermsofthomspaces}
For any smooth proper $k$-variety $X$ of dimension $d$, and for any separable, finitely generated extension $L/k$ there is a canonical isomorphism
\[
\H_0^{s\aone}(X)(L) \isomto H^{n + d}_{\Nis}(Th(\nu_L),\K^{MW}_{n + d}),
\]
functorial with respect to field extensions.
\end{thm}

\begin{proof}
First, combining Chow's lemma with Proposition \ref{prop:birationalinvariance}, we can (and will) assume that $X$ is a smooth projective $k$-variety.  We then proceed in a fashion mirroring the proof of Lemma \ref{lem:suslinchow}:
\[
\begin{split}
\H_0^{s\aone}(X)(L) &\stackrel{(a)}{=} \hom_{\Daone(k)}(\Z,C_*^{s\aone}(X_L)) \\
&\stackrel{(b)}{=} \hom_{\Daone(L)}(\Z,C_*^{s\aone}(X_L)) \\
&\stackrel{(c)}{=} \hom_{\Daone(L)}(C_*^{s\aone}(X_L)^{\vee},\Z) \\
&\stackrel{(d)}{=} \hom_{\Daone(L)}(\tilde{C}_*^{s\aone}(Th(-\tau_{X_L})),\Z) \\
&\stackrel{(e)}{=} \hom_{\Daone(L)}(\tilde{C}_*^{s\aone}(Th(\nu_L)),\Z\langle n+d\rangle[2(n+d)])\\
&\stackrel{(f)}{=} {\mathbb H}^{2(n+d)}_{\Nis}(Th(\nu_L),\Z\langle n+d \rangle) \\
&\stackrel{(g)}{=} H^{n+d}_{\Nis}(Th(\nu_L),\K^{MW}_{n+d})
\end{split}
\]
Identification (a) is the definition of $\aone$-homology.  Identification (b) is a consequence of Proposition \ref{prop:basechange}.  Identification (c) is Lemma \ref{lem:existenceofduals}.  Identification (d) is a consequence of Atiyah duality in the $\aone$-derived category, i.e., Proposition \ref{prop:atiyahduality}.  Identification (e) follows from the discussion of Example \ref{ex:thomspaceofavirtualbundle}.  Identification (f) is a consequence of Corollary \ref{cor:stablerepresentability}.  Finally, identification (g) is a consequence of Corollary \ref{cor:enhancedmotiviccohomologyofthomspaces}.  Functoriality with respect to field extensions is evident from the constructions of all the isomorphisms.
\end{proof}

\begin{prop}
The morphism $T^{n+d} \to Th(\nu)$ defined by Voevodsky induces a map
\[
\hom_{\Daone(k)}(C_*^{s\aone}(Th(\nu)),\Z\langle n + d \rangle[2(n+d)]) \longrightarrow \hom_{\Daone(k)}(C_*^{s\aone}({\pone}^{\wedge n+d}),\Z\langle n + d\rangle[2(n+d)])
\]
that coincides with the pushforward map $\H_0^{s\aone}(X) \to \H_0^{s\aone}(\1_k)$.
\end{prop}

\begin{proof}
This is a consequence of the fact that this map induces the duality map.
\end{proof}

\section{Twisted Chow-Witt groups and Thom isomorphisms}
\label{s:twistedthomisomorphisms}
Suppose $X$ is a smooth $k$-scheme, and $\xi: E \to X$ is a vector bundle of rank $n$.  At the end of Section \ref{s:duality}, we considered groups of the form $H^i_{\Nis}(Th(\xi),\K^{MW}_j)$.  The first goal of this section, achieved at the end of Section \ref{ss:thomisomorphisms} is to describe these groups more concretely in terms of sheaf cohomology on $X$.  To achieve this, we begin by reviewing the theory of twisted Chow-Witt groups in Section \ref{ss:twistedchowwittgroups}; here the twist refers to a choice of line bundle on $X$, which, as we explain below, can be thought of as an ``$\aone$-local system on $X$."  We then prove a twisted Thom isomorphism (Theorem \ref{thm:thomiso}) that identifies cohomology of Milnor-Witt K-theory sheaves on the Thom space of a vector bundle in terms of twisted Chow-Witt groups.  Much of the work here consists of unpacking definitions, but to connect with the work of previous sections, we need to identify the largely formal definition of Milnor-Witt K-theory sheaves given in Definition \ref{defn:stableaonehomotopygroupsofspheres} with something concretely computable; as we explain, this identification uses Morel's computation of the stable $\aone$-homotopy groups of spheres, and the theorem of Orlov-Vishik-Voevodsky proving Milnor's conjecture on quadratic forms.  Thus, taken together, these identifications can be viewed as the ``reduction to Morel's computations" alluded to in the introduction.

The twisted Thom isomorphism we use constitutes a translation of Atiyah's classical theory of Thom isomorphisms for not necessarily oriented manifolds \cite{Atiyah}.  We briefly recall this; suppose $M$ is a smooth manifold.  If $\xi: E \to M$ is a rank $n$ vector bundle over $M$, then $\xi$ is classified by a map $M \to BO(n)$.  One knows that $\pi_1(BO(n)) = \pi_0(O(n)) = \Z/2$ induced by the determinant.  The classifying map of $\xi$ thus induces a homomorphism $\pi_1(M) \to \Z/2$, i.e., an orientation character of $M$, and only depends on the determinant of $\xi$.  The orientation character therefore induces a rank $1$ local system $\Z[\det \xi]$ on $M$.  One can define a Thom class $\tau(\xi)$ in $H^{n}(Th(M),\Z[\det \xi^{\vee}])$.  The Thom isomorphism theorem can be phrased as saying that the
the cup product
\[
\tau(\xi) \cup \cdot: H^i(M,\Z[\det \xi]) \isomto H^{n+i}(Th(M),\Z).
\]
induces a {\em canonical} isomorphism between source and target.

After these preliminaries, Sections \ref{ss:zerothsheaf} and \ref{ss:rationalpointsuptostablehomotopy} are devoted to the proofs of all the results stated in the introduction.  The reader should be warned that throughout this section, unless otherwise mentioned, the base field $k$ will be assumed to have {\em characteristic unequal to $2$}; this requirement is imposed by the theory of Chow-Witt groups via its dependence on Balmer's Witt theory where the requirement that $2$ be invertible is imposed from the beginning.


\subsection{Twisted Chow-Witt groups}
\label{ss:twistedchowwittgroups}

\subsubsection*{The Gersten resolution for a strictly $\aone$-invariant sheaf}
\begin{prop}\label{prop:gerstenresolution}
Suppose $\mathscr{M}$ is a strictly $\aone$-invariant sheaf of groups and $k$ is an infinite field.  If $X$ is a smooth $k$-variety, then the restriction of $\mathscr{M}$ to the small Zariski site of $X$, which we denote $\mathscr{M}|_{X_{Zar}}$ admits the following resolution:
\[
\mathscr{M}|_{X_{Zar}} \longrightarrow \bigoplus_{x \in X^{(0)}} \mathscr{M}(\kappa(x)) \longrightarrow \bigoplus_{x \in X^{(1)}} \mathscr{M}_{-1}(\kappa_{x}) \longrightarrow \cdots \longrightarrow \bigoplus_{x \in X^{(i)}} \mathscr{M}_{-i}(\kappa_x) \longrightarrow \cdots.
\]
\end{prop}

\begin{proof}
Since $\mathscr{M}$ is strictly $\aone$-invariant, we know that $\mathscr{M}|_{X_{Zar}}$ admits a Cousin resolution, furthermore we know that all cohomology presheaves in the Zariski and Nisnevich topology agree (these statements are a consequence of \cite[Theorem 5.1.10 and Proposition 5.3.2a]{CTHK} and \cite[Theorem 8.3.1]{CTHK} respectively).  We need to show that $H^{q}_x(X,\mathscr{M}) = \mathscr{M}_{-q}(\kappa_x)$.  By definition, the group on the left hand side is $\colim H^q_{\bar{x} \cap U}(U,\mathscr{M}|_U)$.  For any smooth closed immersion $Z \hookrightarrow X$, we have $H^q_Z(X,\mathscr{M}) = H^q_{\Nis}(Th(N_{Z/X}),\mathscr{M})$ by the purity isomorphism.  By shrinking $X$ and $Z$ and using Nisnevich excision, we can assume $N_{Z/X}$ is trivial and the Lemma \ref{lem:contractionshift} shows that $H^q(Th(N_{Z/X}),\mathscr{M}) = H^{q-codim{Z/X}}(X,\mathscr{M})_{-codim{Z/X}}$.
\end{proof}

\begin{lem}[{\cite[Exercise 23.4]{MVW}}]
\label{lem:contractionshift}
If $\mathscr{M}$ is strictly $\aone$-invariant in the Zariski topology, for any smooth scheme $X$ we have a canonical isomorphism
\[
H^m_{X \times \{0\}}(X \times {\mathbb A}^m,\mathscr{M}) \isomto H^{m-n}(X,\mathscr{M})_{-n}
\]
\end{lem}


\begin{thm}
\label{thm:milnorwittgersten}
For any smooth scheme $X$, the sheaves $\K^{MW}_n$ admit a resolution of the form
\[
\K^{MW}_n|_X \longrightarrow \bigoplus_{x \in X^{(0)}}\K^{MW}_n(\kappa_x) \longrightarrow \cdots \longrightarrow \bigoplus_{x \in X^{(i)}} \K^{MW}_{n-i}(\kappa_x) \longrightarrow \cdots;
\]
we write $C^*(X,K^{MW}_n)$ for this complex.
\end{thm}

\begin{proof}
This is immediate from Proposition \ref{prop:gerstenresolution} and Proposition \ref{prop:contractionsofmilnorwittsheaves}.
\end{proof}

\begin{defn}
\label{defn:twistedidealgersten}
We recall from \cite[Section 9]{Fasel1} that for each regular scheme $X$, line bundle $\L$ on $X$ and integer $j$ we have a complex $C^*(X,I^j,\L)$ which in degree $k$ is
\[
C^k(X,I^j,\L) = \bigoplus_{x \in X^{(k)}} I^j(\mathcal{O}_{X,x},\L_x),
\]
where $I^j(\mathcal{O}_{X,x},\L_x)$ is the $j$-th power of the fundamental ideal of the Witt ring of modules of finite length with duality given by the line bundle $\L$ at the point $x$.
In this way we can clearly define a complex of sheaves on $X_{Nis}$, which will be denoted by $C^*(-,I^j,\L)$.
\end{defn}

\begin{rem}
The reader should be warned that Fasel uses slightly different notation, and different indexing, for these complexes.
\end{rem}

\begin{defn}
The Gersten resolution of the sheaf $\K^M_n$ of unramified Milnor $K$-theory is of the form
\[
\K^{M}_n|_X \longrightarrow \bigoplus_{x \in X^{(0)}}\K^{M}_n(\kappa_x) \longrightarrow \cdots \longrightarrow \bigoplus_{x \in X^{(i)}} \K^{M}_{n-i}(\kappa_x) \longrightarrow \cdots,
\]
and similarly for the sheaves $\K^M_n/2$. In light of Proposition \ref{prop:gerstenresolution} the existence of this resolution follows from the fact (see \cite{MIntro}) that $\K^M_*$ is a homotopy module.  We write $C^*(-,K^M_n)$ and $C^*(-,K^M_n/2)$ for these Gersten complexes of sheaves.
\end{defn}


As described in \cite[Section 10]{Fasel1}, there is a map of complexes of sheaves from $C^*(-,I^j,\L)$ to $C^*(-,I^j/I^{j+1},\L)$ and also a map from $C^*(X,K^M_j/2) \to C^*(X,I^j/I^{j+1},\L)$.  The aforementioned map is an isomorphism by the Orlov-Vishik-Voevodsky theorem proving Milnor's conjecture on quadratic forms; see \cite[Theorem 4.1]{OrViVo}, or \cite[Theorem 1.1 and \S 2.3]{MMilnor} for a discussion in the spirit of this paper.  As a consequence, we can compute the cohomology sheaves of the complexes $C^*(-,I^j,\L)$.

\begin{prop}\label{prop:twistedidealgersten}
The complex of sheaves $C^*(-,I^j,\L)$ has the following cohomology:
\begin{enumerate}
\item $H^i C^*(-,I^j,\L) = 0$ for $i < 0$,
\item $H^0 C^*(-,I^j,\L) = \I^j(\L)^{nr}$ (the unramified sheaf associated to $I^j$) for $i = 0$, and
\item $H^i C^*(-,I^j,\L) = 0$ for $i > 0$.
\end{enumerate}
\end{prop}

\begin{proof}
The first two assertions are immediate from the definition. Since we are dealing with a local problem, we may assume that $\L$ is trivial. First consider the case where $j \leq 0$. In this case the complex in question is, by definition, simply the Gersten-Witt complex of Balmer-Walter (see \cite[Theorem 7.2]{BalmerWalter}), which is a resolution. We proceed by induction on $j$, employing the short exact sequence of complexes (see \cite[D\'efinition 9.2.10]{Fasel1})
\[
0\longrightarrow C^*(-,I^j,\L) \longrightarrow C^*(-,I^{j-1},\L) \longrightarrow C^*(-,I^{j-1}/I^j,\L) \to 0.
\]
It follows from the Orlov-Vishik-Voevodsky theorem (references just prior to the statement of the proposition) that the third complex is isomorphic to the Gersten complex of Milnor $K$-theory $\K^M_{j-1}/2$, and hence a resolution; and the middle complex is a resolution by induction. Therefore the first complex is a resolution as well.
\end{proof}

The next result ties together all of the definitions made so far.  For context, recall that we defined Milnor-Witt K-theory sheaves in terms of zeroth stable $\aone$-homotopy sheaves of spectra (Definition \ref{defn:stableaonehomotopygroupsofspheres}).  This definition has the benefit of explaining why the sheaves $\K^{MW}_*$ are strictly $\aone$-invariant and give rise to a homotopy module (see Definition \ref{defn:homotopymodule}).  We observed in Theorem \ref{thm:milnorwittgersten} that this structure is sufficient to give a Gersten resolution for $\K^{MW}_n$.  To connect this {\em a priori} abstract description to the discussion of Fasel that we have reviewed above requires Morel's explicit description of the sheaves $\K^{MW}_n$.

\begin{thm}[Morel]
\label{thm:morelfaselcomparison}
Assume $k$ is an infinite perfect field of characteristic unequal to $2$.  For any smooth scheme $X$, there is a (functorial in $X$) quasi-isomorphism of complexes
\[
C^*(-,K_j^{MW}) = C^*(-,K_j^M)\times_{C^*(-,I^j/I^{j+1},\O_X)} C^*(-,I^j,\O_X).
\]
\end{thm}

\begin{proof}
Consider the zeroth cohomology sheaf $H^0(C^*(-,K_j^M)\times_{C^*(-,I^j/I^{j+1},\O_X)} C^*(-,I^j,\O_X))$.  In \cite[Theorem 5.3]{MPUissance}, Morel gives an explicit description of the sections of the aforementioned cohomology sheaf over finitely generated extensions $L/k$ in terms of what are called Milnor-Witt K-theory of fields; this implicitly uses the Orlov-Vishik-Voevodsky theorem proving Milnor's conjecture on quadratic forms.  This terminology is justified by \cite[Theorem 5.39 and Remark 5.41]{MField}, which implicitly uses the assumption that $k$ is perfect, and asserts that these Milnor-Witt K-theory of fields are precisely the sections of what we called Milnor-Witt K-theory sheaves over fields (recall Definition \ref{defn:stableaonehomotopygroupsofspheres}).

Moreover, there is a morphism of sheaves from $\K^{MW}_j$ to the zeroth cohomology sheaf above; see \cite[\S 2.2]{MMilnor} for an explanation of this construction.  Since the zeroth cohomology sheaf considered above is strictly $\aone$-invariant by its very construction, the discussion of the previous paragraph shows that the aforementioned morphism is an isomorphism.  Thus, both complexes are (functorially in $X$) resolutions of the sheaf $\K^{MW}_j$ and are therefore quasi-isomorphic.
\end{proof}


Again following Fasel, we define (with slightly different notation and indexing):

\begin{defn}\label{defn:chowwittcomplex}
The $j$-th Chow-Witt complex of $X$ with twist $\L$ is the pull-back complex of sheaves
\[
C^*(-,K_j^{MW},\L) := C^*(-,K_j^M)\times_{C^*(-,I^j/I^{j+1},\L)} C^*(-,I^j,\L).
\]
\end{defn}


\begin{thm}\label{thm:twistedmilnorwittresolution}
The only non-vanishing cohomology sheaf of the complex of sheaves $C^*(-,K_j^{MW},\L)$ occurs in degree zero.
\end{thm}

\begin{proof}
This is a straightforward consequence of Proposition \ref{prop:twistedidealgersten}.
\end{proof}

\begin{defn}\label{defn:twistedmilnorwitt}
Given a regular scheme $X$ and a line bundle $\L$ on $X$, we define the $\L$-twisted Milnor-Witt $K$-theory sheaves on $X$ by means of the formula:
\[
\K^{MW}_j(\L) := H^0 C^*(-,K_j^{MW},\L).
\]
\end{defn}


By Theorem \ref{thm:twistedmilnorwittresolution}, $C^*(-,K^{MW}_j,\L)$ is a flasque resolution of the sheaf $\K^{MW}_j(\L)$, so that we deduce the following result.

\begin{cor}\label{cor:milnorwittcohomology}
For any closed subscheme $Y\subseteq X$ with open complement $U = X - Y$, we have an isomorphism
\[
H^*_Y(X,\K^{MW}_j(\L)) = H^*(\mathrm{ker}\{ C^*(X,K^{MW}_j,\L) \longrightarrow C^*(U,K^{MW}_j,\L)\}).
\]
\end{cor}

\begin{defn}
\label{defn:quadraticzerocycles}
For a smooth proper $k$-variety $X$ of dimension $d$, we define the Chow-Witt group of quadratic zero cycles as
\[
\CH_0(X) := H^{d}_{\Nis}(X,\K^{MW}_d(\omega_X)).
\]
Note that, by the Gersten resolution, $\CH_0(X)$ is a quotient of $\bigoplus_{x \in X^{(d)}} \K^{MW}_0(\kappa_x/k)$.
\end{defn}

\begin{rem}\label{rem:ChowWittgroups}
Theorem \ref{thm:twistedmilnorwittresolution} immediately implies that there is a canonical isomorphism
\[
H^p_{\Nis}(X,\K^{MW}_p(\L)) = \CH^p(X,\L),
\]
where the right-hand side is the twisted Chow-Witt group defined by Fasel in \cite[D\'efinition 10.2.16]{Fasel1}.  In particular, for a smooth proper variety $X$, we have a canonical identification $\CH_0(X) = \CH^d(X,\omega_X)$, and this explains our choice of notation.
\end{rem}




\subsection{The Thom isomorphism theorem}
\label{ss:thomisomorphisms}
\subsubsection*{Recollection: functoriality.}
Fasel defines (in \cite[Section 10.4]{Fasel1} and \cite[Sections 5 through 7]{Fasel2}) pullbacks, products and proper pushforwards for the complexes $C^*(-,K^{MW}_j, \L)$ as follows:

\begin{thm}
\label{thm:milnorwittresfunctoriality}
Let $X$, $Y$ and $Z$ be regular schemes and suppose $f:X\to Y$ is a morphism and $g:X\to Z$ is proper of relative dimension $n$. Then:
\begin{enumerate}

\item Assume $f$ is flat. For any line bundle $\L$ on $Y$, we have pullback homomorphisms
\[
f^*: C^*(Y,K^{MW}_j,\L) \longrightarrow C^*(X,K^{MW}_j, f^* \L)
\]
which are functorial for composition of flat morphisms.

\item For $f$ arbitrary, for all $i$, and for any line bundle $\L$ on $Y$, we have functorial Gysin pullbacks
\[
f^!: H^i(C^*(Y,K^{MW}_j,\L)) \longrightarrow H^i(C^*(X,K^{MW}_j, f^* \L)).
\]
If $f$ is flat, then $f^!$ coincides with the map induced on cohomology by the flat pullback $f^*$.

\item For any line bundle $\L$ on $Z$ we have pushforward homomorphisms
\[
g_*: C^*(X,K^{MW}_j,g^* \L\otimes\omega_X) \longrightarrow C^{*-n}(Z,K^{MW}_{j-n},\L\otimes\omega_Y)
\]
which are functorial for composition of proper morphisms.

\item Let $\L$ and $\L^\prime$ be line bundles on $X$. There is a functorial (with respect to the Gysin pullbacks) graded associative product
\[
H^*(C^*(X,K^{MW}_j, \L))\otimes H^*(C^*(X,K^{MW}_k, \L^\prime))\longrightarrow H^*(C^*(X,K^{MW}_{j+k}, \L\otimes \L^\prime))
\]
defined from an exterior product via pullback along the diagonal.
\end{enumerate}
\end{thm}

\subsubsection*{Interlude: the projection formula}
We will need a projection formula for the cohomology of the Chow-Witt complexes.  In some of the results below we will, for notational convenience, suppress the line bundle twists that are implicit; to recover the twists relevant to a given diagram, we refer the reader to Theorem \ref{thm:milnorwittresfunctoriality}.  The crucial ingredient in the proof is a proper base change theorem; we thank Jean Fasel for showing us the proof of the following result.

\begin{prop}
\label{prop:properbasechange}
If
\[
\xymatrix{
{X^\prime} \ar[r]^v\ar[d]_g & {X} \ar[d]^f \\
 {Y^\prime} \ar[r]_u & {Y}
}
\]
is a cartesian square of smooth schemes with $f$ proper, then $u^! f_* = g_* v^!$.
\end{prop}

\begin{proof}
The proof follows the same pattern as \cite[Proposition 12.5]{RostChow}.  Any morphism of smooth schemes can be factored as a closed immersion followed by a smooth (in particular, flat) morphism.  Then, we contemplate the diagram
\[
\xymatrix{
X' \ar[r]^{\Gamma_v}\ar@{=}[d] & X' \times X \ar[r]^{p_X}\ar[d]^{g \times id} & X \ar@{=}[d] \\
X' \ar[r]\ar[d]^{g} & Y' \times X \ar[r]\ar[d]^{id \times f}& X \ar[d]^{f} \\
Y' \ar[r]^{\Gamma_u} & Y' \times Y \ar[r]^{p_Y} & Y.
}
\]
Using the fact that the Gysin pullback coincides with the usual flat pullback for flat morphisms \cite[Proposition 7.4]{Fasel2}, the functoriality of pullbacks \cite[Theorem 5.11]{Fasel2} reduces us to proving the result in the case where $u$ is flat or a closed immersion.

If $u$ is flat, this is \cite[Corollaire 12.3.7]{Fasel1}.  For a closed immersion, it follows, in light of the naturality of the deformation to the normal cone construction, from the way the Gysin map is defined (see \cite[Definition 5.5]{Fasel2}) and another application of \cite[Corollaire 12.3.7]{Fasel1}.

\end{proof}

\begin{cor}
\label{cor:projectionformula}
Suppose $f: X \to Y$ is a proper morphism of smooth schemes.  For any classes $\eta, \eta'$ (with arbitrary twists) we have a canonical identification
\[
f_*(f^! \eta \cup \eta') = \eta \cup f_*\eta'.
\]
\end{cor}

\begin{proof}
Given base change, the proof is standard. Applying Proposition \ref{prop:properbasechange} to the square
\[
\xymatrix{
{X} \ar[r]^{\Gamma_f^t}\ar[d]_f & {Y\times X} \ar[d]^{1\times f} \\
 {Y} \ar[r]_{\Delta_Y} & {Y\times Y},
}
\]
we conclude that $\Delta^!_Y (1\times f)_* = f_* (\Gamma_f^t)^! = f_* \Delta^!_X (f\times 1)^*$ where the second equality comes from the equation $\Gamma_f^t = (f\times 1) \Delta_X$.
\end{proof}

\begin{rem}
The product structure on the cohomology of Chow-Witt complexes is not (graded) commutative (cf. \cite[Remark 6.7]{Fasel2}). Nevertheless the above proof clearly applies {\em mutatis mutandis} to show that $f_*(\eta \cup f^! \eta') = f_*\eta \cup \eta'.$
\end{rem}

\subsubsection*{Thom isomorphisms}
Suppose $\xi: E \to X$ is a vector bundle with zero section $i: X \to E$.  Let $\det \xi$ be the determinant bundle.  With the definition of $\det \xi^{\vee}$-twisted Chow-Witt groups above, we have the following purity result; see \cite[Remarque 10.4.8]{Fasel1}.

\begin{prop}\label{prop:milnorwittlocalization}
For a regular subscheme $Y\subseteq X$ of codimension $c$ with normal bundle $\nu$ and complement $U = X - Y$, there is a quasi-isomorphism
\[
C^{*-c}(Y,K^{MW}_j,\det\nu)\longrightarrow \mathrm{ker}\{ C^*(X,K^{MW}_j) \longrightarrow C^*(U,K^{MW}_j)\}.
\]
\end{prop}

\begin{defn}\label{defn:thomclass}
Let $X$ be a smooth scheme and $\xi: E\to X$ a vector bundle of rank $r$ and with zero section $i: X\to E$. By Proposition \ref{prop:milnorwittlocalization}, we have an isomorphism $i_*: H^0_{\Nis}(X,\K^{MW}_0) \to H^r_{i(X)}(E,\K^{MW}_r(\det \xi^{\vee}))$.

The image $\tau(\xi) = i_*(1_X)$ of $1_X$ under this isomorphism is called the Thom class of $\xi$. Pull back along $\xi$ and cup product with the Thom class defines homomorphisms $\tau_{p,q}(\xi): H^p_{\Nis}(X,\K^{MW}_q(\det \xi)) \to H^{p+r}_{i(X)}(E,\K^{MW}_{q+r})$, and the latter group is canonically isomorphic (since $E - i(X) \to E \to Th(\xi)$ is a cofibration sequence, {\it cf.} Lemma \ref{lem:cohomologicaldimensionofthomspaces}) to $H^{p+r}_{\Nis}(Th(\xi),\K^{MW}_{q+r}).$
\end{defn}

\begin{thm}[Thom isomorphism]
\label{thm:thomiso}
The bigraded $H^*_{\Nis}(X,\K^{MW}_*)$-module $H^*_{i(X)}(E,\K^{MW}_*(\det \xi^{\vee}))$ is free of rank one on the Thom class $\tau(\xi)$. Consequently we have a collection of Thom isomorphisms $\tau_{p,q}(\xi): H^p_{\Nis}(X,\K^{MW}_q(\det \xi)) \to H^{p+r}_{\Nis}(Th(\xi),\K^{MW}_{q+r})$ defined by pullback followed by cup-product with the Thom class.
\end{thm}

\begin{proof}
Let $\xi: E \to X$ be a vector bundle on $X$.  Consider the following diagram:
\[
\xymatrix{
H^p_{\Nis}(X,\K^{MW}_q(\det \xi)) \ar[r]^{i_*} \ar[d]^{\xi^*} & H^{p+r}_{i(X)}(E,\K^{MW}_{q+r}) \\
H^p_{\Nis}(E,\K^{MW}_q(\xi^* \det \xi)) \ar[ur]^{\cup \tau(\xi)}.
}
\]
The top horizontal morphism is an isomorphism by Proposition \ref{prop:milnorwittlocalization}.

By \cite[Proposition 6.8]{Fasel2}, the fundamental class $1_X$ is both a left and right unit.  Moreover, $i^!\xi^* = id$ by \cite[Lemma 5.10]{Fasel2}.  Therefore, for any class $\eta$, we have $\eta = i^!\xi^*\eta \cup 1$.  Combining the observation just made with the projection formula \ref{cor:projectionformula}, we then have
\[
i_*(\eta) = i_*(i^!\xi^*\eta \cup 1) = \xi^*\eta \cup i_*(1),
\]
that is, the diagram commutes.

\end{proof}

\subsection{The zeroth stable $\aone$-homotopy sheaf}
\label{ss:zerothsheaf}
The goal of this section is to prove Theorems \ref{thmintro:main} and \ref{thmintro:compatibilities} from the introduction.

\begin{thm}
\label{thm:main}
Suppose $k$ is an infinite perfect field having characteristic unequal to $2$, and $X$ is a smooth proper $k$-variety. For any separable, finitely generated extension $L/k$, there are isomorphisms
\[
\bpi_0^s(\Sigma^{\infty}_{\pone}X_+)(L) \isomto \CH_0(X_L),
\]
functorial with respect to field extensions.
\end{thm}

\begin{proof}
By Theorem \ref{thm:ponestablehurewiczisomorphism}, we only need exhibit a natural isomorphism $\H_0^{s\aone}(X)(L)\isomto \CH_0(X_L)$. We can assume $X$ has dimension $d$.  By Theorem \ref{thm:aonehomologyintermsofthomspaces} we have a canonical isomorphism (functorial with respect to field extensions)
\[
\H_0^{s\aone}(X)(L) \isomto H^{n + d}_{\Nis}(Th(\nu_L),\K^{MW}_{n + d}),
\]
where $\nu: V\to X$ is a vector bundle of rank $n$ such that $[T_X \oplus V] = n+d$ in $K_0(X)$. By the Thom isomorphism \ref{thm:thomiso}, the right hand side is canonically isomorphic to $H^d_{\Nis}(X_L,\K_d^{MW}(\det \nu_L))$, and by properties of the determinant, $\det \nu_L = \omega_{X_L}$ (the canonical bundle). Finally, $H^d_{\Nis}(X_L,\K_d^{MW}(\omega_{X_L})) = \CH_o(X_L)$ by definition.
\end{proof}

\begin{thm}
\label{thm:compatibilities}
The isomorphisms of \textup{Theorem \ref{thm:main}} satisfy the following compatibilities.
\begin{itemize}
\item[A)] The ``Hurewicz" homomorphism $\bpi_0^s(\Sigma^{\infty}_{\pone}X_+) \to \H_0^S(X)$ induces the forgetful map $\CH_0(X_L) \to CH_0(X_L)$ upon evaluation on sections over a separable finitely generated extension $L/k$.
\item[B)] The pushforward map $\bpi_0^s(\Sigma^{\infty}_{\pone}X_+)(L) \to \bpi_0^s(\Sigma^{\infty}_{\pone}\Spec k_+)(L)$ induces by means of the Thom isomorphism \textup{Theorem \ref{thm:thomiso}} and Morel's identification of $\bpi_0^s(\Sigma^{\infty}_{\pone}\Spec k_+)(L)$ with $GW(L)$ a morphism $\tdeg: \CH_0(X_L) \to \CH_0(\Spec L)$ that coincides with the corresponding pushforward defined by Fasel.
\item[C)] The two identifications just mentioned are compatible, i.e., the diagram
\[
\xymatrix{
\CH_0(X_L) \ar[r]\ar[d]^-{\tdeg} & CH_0(X_L) \ar[d]^-{\deg} \\
GW(L) \ar[r]^-{\operatorname{rk}} & \Z
}
\]
commutes.
\end{itemize}
\end{thm}

\begin{proof}
\noindent A) The homomorphism in the statement is the map from Proposition \ref{prop:stabletosuslinhomology} composed with the stable Hurewicz homomorphism $\bpi_0^s(\Sigma^{\infty}_{\pone}X_+) \to \H_0^{s\aone}(X)$ of \ref{eqn:hurewiczhomomorphism}. Now, combine Lemma \ref{lem:suslinchow} and Theorem \ref{thm:main}.\\
\noindent B) Combining Theorem \ref{thm:aonehomologyintermsofthomspaces} and Theorem \ref{thm:thomiso}, evaluation of the pushfoward map $\H_0^{s\aone}(X) \to \H_0^{s\aone}(\Spec k)$ on sections over a finitely generated, separable extension $L/k$ gives rise to a morphism
\[
H^d_{\Nis}(X_L,\K^{MW}_d(\omega_X)) \to H^0_{\Nis}(\Spec L,\K^{MW}_0).
\]
By Theorem \ref{thm:main} and \cite[Lemma 2.10]{MField}, this morphism can be viewed as a morphism $\CH_0(X_L) \to GW(L)$; we want to know that this coincides with the pushforward defined by Fasel.

Unwinding the definitions, to identify the two pushforwards, we have to prove commutativity of the diagram
\begin{equation}
\label{eqn:commutativity}
\xymatrix{
H^{d}(X_L,\K^{MW}_d(\omega_{X_L})) \ar[d]\ar[r]^{\tdeg} & H^0(\Spec L,\K^{MW}_0) \ar[d] \\
H^{n+d}(Th(\nu_L),\K^{MW}_{n+d}) \ar[r] & H^{n+d}(T^{n+d}_L,\K^{MW}_{n+d}),
}
\end{equation}
where the left vertical arrow is the Thom isomorphism, the right vertical arrow is the suspension isomorphism, and the lower horizontal arrow is the morphism induced by Voevodsky's duality theorem \ref{thm:voevodskyduality}. To simplify of notation, we may use base change \ref{prop:basechange} and assume $k = L$ in what follows.\newline
\noindent {\em Step 1.}  We claim it suffices to prove the result for the pushforward induced by a finite field extension $\Spec F \to \Spec k$.  To see this, observe that the morphism $\tdeg$ has a local definition. Indeed, there are transfer homomorphisms
\[
tr_x: GW(\kappa_x,\omega_{\kappa_x/k}) \to GW(k)
\]
and the map $\tdeg$ is induced by the homomorphism $\bigoplus_{x \in X^{(d)}} tr_x$; this follows by combining the local definition of $\operatorname{deg}$ for Chow groups with the study of the push-forward for Witt groups in \cite[\S 6.4]{Fasel1} together with the identification $GW(F) \cong W(F) \times_{\Z/2} \Z$. We note here for later use that the map $tr_x$ is a Scharlau transfer with respect to the field trace of $\kappa_x$ over $k$.

On the other hand, collapsing the complement of $x$ in $X$ gives a morphism $X \to X/(X - x).$  Since the normal space to $x$ is trivial, the space $X/(X-x)$ is isomorphic to $T^{\wedge d}$; to fix such an isomorphism we need a choice of trivialization.  Keeping track of the twist by $\omega_{X}$, the collapse map just mentioned gives rise to a homomorphism
\[
H^d_{\Nis}(X/X - x,\K^{MW}_d(\omega_X)) \longrightarrow H^{d}_{\Nis}(X,\K^{MW}_{d}(\omega_{X})),
\]
where the group on the left hand side is defined as the $d$-th cohomology of the cone of the morphism $C^*(X - x,K^{MW}_*,\omega_{X-x}) \to C^*(X,K^{MW}_*,\omega_X)$; the resulting complex is supported at $x$ by its very construction.\newline
We claim there is a suspension isomorphism $H^d(X/X - x,\K^{MW}_d(\omega_X)) \isomt GW(\kappa_x,\omega_{\kappa_x/k})$.  Indeed, by inspection of the Gersten resolution (the only term that appears is $K^{MW}_0(\kappa_x,\omega_{\kappa_x/k})$), this is a consequence of Morel's isomorphism.

Using the fact that $H^{d}_{\Nis}(X,\K^{MW}_{d}(\omega_{X}))$ admits a surjection from a sum of groups of the form $GW(\kappa_x,\omega_{x/k})$, it suffices to prove that $tr_x$ coincides with the homomorphism $GW(\kappa_x,\omega_{\kappa_x/k}) \to GW(k)$ given by composing the morphism of the previous paragraph with the composite of the morphisms from the three unlabeled edges in Diagram \ref{eqn:commutativity} for any $x \in X^{(d)}$.  By functoriality of the Thom isomorphism, this corresponds to proving the initial compatibility in the case where $X = \Spec F$, as claimed. We will write $tr_{F/k}$ for Fasel's pushforward (the Scharlau transfer for the field trace). \newline
\noindent{\em Step 2.}  We now assume that $X = \Spec F$; the duality construction simplifies in this case. Pick an embedding $i: \Spec F \hookrightarrow \pone_k$.  There is an induced map $\pone_k \to Th(\nu_i)$, sometimes called the ``co-transfer map," ({\em cf.} \cite{LevineSlices}) which is the duality map.  We need to show that the composite map
\[
t: H^0_{\Nis}(\Spec F,\K^{MW}_0(\nu_i)) \isomto H^1_{\Nis}(Th(\nu_i),\K^{MW}_1) \to H^1_{\Nis}(\pone_k,\K^{MW}_1) \isomto H^0_{\Nis}(\Spec k,\K^{MW}_0),
\]
where the first map is the Thom isomorphism and the last map is the inverse suspension isomorphism, is precisely the pushforward map $tr_{F/k}$ of Fasel.

Following \cite[\S 3]{MField}, the composite map is precisely Morel's ``cohomological" transfer.  This transfer can be described at the level of the Gersten resolution as follows. Consider $\aone_k$ with coordinate $t$.  There is a short exact sequence of the form
\[
0 \longrightarrow K^{MW}_1(k) \longrightarrow K^{MW}_1(k(t)) \stackrel{\sum_{\nu} \partial_{\nu}}{\longrightarrow} \bigoplus_{\nu} K^{MW}_0(\kappa_{\nu},\omega_{\kappa_{\nu}/k}) \longrightarrow 0.
\]
Here the maps $\partial_{\nu}$ are the residue maps in Milnor-Witt $K$-theory described by Morel in \cite[Theorem 2.15]{MField}, and the short exact sequence is proven in \cite[Theorem 2.24]{MField}.
Picking a primitive element $\theta$ of $F$ with minimal polynomial $f(t)$, we identify $F$ with the closed point $x$ of $\aone_k$ corresponding to $f(t)$. Given an element $\alpha\in K_0^{MW}(F) = K_0^{MW}(\kappa_x)$ the exact sequence above shows that we can lift $\alpha$ along the differential in the Gersten complex to some $\tilde{\alpha}\in K_1^{MW}(k(t))$. In other words, we have the equation $\partial_x^{can}(\tilde{\alpha}) = \alpha$; here $\partial_x^{can}$ is the $x$-component of the differential in the Gersten complex, which is obtained from $\partial_x$ via twist with the canonical bundle. Now $t(\alpha) = -\partial_\infty^{can}(\tilde{\alpha})$. To prove our desired compatibility it therefore suffices to show that $-\partial_\infty^{can} = tr_{F/k}\circ \partial_x^{can}$. Let $\beta\in K_1^{MW}(k(t))$. It is clear from Morel's description of the residues that the ranks of $-\partial_\infty^{can}(\beta)$ and $tr_{F/k}\circ \partial_x^{can}(\beta)$ coincide (the sign comes from the fact that the parameter at $\infty$ is a multiple of $1/t$). It is therefore sufficient to show that the equation $-\partial_\infty^{can} = tr_{F/k}\circ \partial_x^{can}$ holds on (canonically twisted) Witt groups; but that is an immediate consequence of Schmid's reciprocity theorem \cite[2.4.5]{Schmid}.

\noindent C) The Hurewicz maps are functorial by construction.
\end{proof}







\subsection{Rational points up to stable $\aone$-homotopy}
\label{ss:rationalpointsuptostablehomotopy}
Finally, in this section, we prove Theorem \ref{thmintro:rationalpointsuptostablehomotopy}.

\begin{lem}
\label{lem:morphismsofstrictlyaoneinvariantsheaves}
If $f: \mathscr{M} \to \mathscr{M}'$ is a morphism of strictly $\aone$-invariant sheaves of groups, then $f$ is an isomorphism (resp. monomorphism, resp. epimorphism) if and only if it is bijective (resp. injective, resp. surjective) on sections over any finitely generated, separable extension $L/k$.
\end{lem}

\begin{proof}
This is immediate from the Gersten resolution \ref{prop:gerstenresolution}.
\end{proof}

\begin{prop}
\label{prop:epimorphicity}
Suppose $k$ is a perfect field, and $X$ is a smooth, proper $k$-scheme.  The canonical morphism $\H_0^{s\aone}(X) \to \H_0^{S}(X)$ is an epimorphism.
\end{prop}



\begin{proof}
Since both sheaves in question are strictly $\aone$-invariant, by Lemma \ref{lem:morphismsofstrictlyaoneinvariantsheaves} it suffices to prove that the map in question is surjective on sections over finitely generated extension fields $L/k$.  Let $d = \dim X$ and $\nu: V \to X$ the vector bundle of Theorem \ref{thm:voevodskyduality}.

Via duality, the homomorphism $\H_0^{s\aone}(X)(L) \to \H_0^{S}(X)(L)$ is identified with the homomorphism $H^{n+d}_{\Nis}(Th(V),\K^{MW}_{n+d}) \to H^{n+d}_{\Nis}(Th(V),\K^M_{n+d})$ induced by the epimorphism of strictly $\aone$-invariant sheaves $\K^{MW}_{n+d} \to \K^M_{n+d}$. Now, by Lemma
\ref{lem:cohomologicaldimensionofthomspaces} the cohomological dimension of $Th(V)$ is $n+d$ so the functor $H^{n+d}_{\Nis}(Th(V),-)$ is right exact.
\end{proof}

\begin{defn}
Suppose $X$ is a smooth $k$-variety.  We say $X$ {\em has a rational point up to stable $\aone$-homotopy} if the structure map $\Sigma^{\infty}_{\pone}X_+ \to \So_k$ is a split epimorphism.  If $X$ has a rational point up to stable $\aone$-homotopy, a choice of splitting $\So_k \to \Sigma^{\infty}_{\pone}X_+$ is called a rational point up to stable $\aone$-homotopy.
\end{defn}

\begin{lem}
\label{lem:rationalpointuptostablehomotopyequivalentconditions}
Suppose $X$ is a smooth $k$-variety.  The following conditions are equivalent.
\begin{itemize}
\item[i)] The variety $X$ has a rational point up to stable $\aone$-homotopy.
\item[ii)] The structure map $\H_0^{s\aone}(X) \to \H_0^{s\aone}(\Spec k)$ is a split epimorphism.
\item[iii)] There exists an element $x \in \bpi_0^{s}(X)(k)$ lifting $1 \in \bpi_0^{s}(\So_k)$.
\item[iv)] There exists an element $x \in \H_0^{s\aone}(X)(k)$ lifting $1 \in \H_0^{s\aone}(X)(k)$.
\end{itemize}
\end{lem}

\begin{proof}
The equivalences (i) $\Leftrightarrow$ (ii) and (iii) $\Leftrightarrow$ (iv) are consequences of the stable Hurewicz theorem; one only needs to note that the Hurewicz isomorphism is functorial in the input space.  To see that (i) $\Leftrightarrow$ (iii), observe that elements $x \in \bpi_0^{s}(X)(k)$ are precisely morphisms of sheaves $\So_k \to \bpi_0^{s}(\Sigma^{\infty}_{\pone}X_+)$.  Under this identification, the element $1 \in \bpi_0^{s}(\So_k)$ corresponds to the identity morphism $\So_k \to \So_k$.
\end{proof}

\begin{thm}
\label{thm:stablerationalpointnonformallyrealcase}
Assume $k$ is an infinite perfect field (having characteristic unequal to $2$).  A smooth proper $k$-variety $X$ has a rational point up to stable $\aone$-homotopy if and only if $X$ has a $0$-cycle of degree $1$.
\end{thm}

\begin{proof}
Consider the diagram
\[
\xymatrix{
\H_0^{s\aone}(X) \ar[r]\ar[d] & \H_0^{S}(X) \ar[d] \\
\H_0^{s\aone}(\Spec k) \ar[r] & \H_0^{S}(\Spec k).
}
\]
By Proposition \ref{prop:epimorphicity} both horizontal arrows are epimorphisms.  By Lemma \ref{lem:morphismsofstrictlyaoneinvariantsheaves} it follows that upon taking sections over $k$ the induced horizontal maps are still epimorphisms.  Moreover, by Theorems \ref{thm:main} and \ref{thm:compatibilities}, the left vertical map coincides with the pushforward map $\tdeg: \CH_0(X) \longrightarrow GW(k)$.

If $X$ has a rational point up to stable homotopy, by the equivalent conditions of Lemma \ref{lem:rationalpointuptostablehomotopyequivalentconditions} we know that there exists an element $x \in \H_0^{s\aone}(X)(k)$ lifting $1 \in K^{MW}_0(k)$.  Since $1 \in \H_0^{s\aone}(\Spec k)(k)$ is sent to $1$ in $\H_0^{S}(X)(k)$, it follows immediately that the image of $x$ in $\H_0^S(X)(k)$ is mapped to a $0$-cycle of degree $1$.

Conversely, suppose $X$ has a $0$-cycle of degree $1$.  By definition, the map $\H_0^S(X) \to \Z$ is a split epimorphism.  Let $x \in \H_0^S(X)(k)$ be the corresponding lift of $1$.  Since the map $\H_0^{s\aone}(X)(k) \to \H_0^{S}(X)(k)$ is surjective, it follows that there exists $\tilde{x}$ in $\H_0^{s\aone}(X)(k)$ lifting this element. By assumption, this element is not in the kernel of the induced homomorphism $\H_0^{s\aone}(X)(k) \to K^{MW}_0(k)$.  Thus, let  $\bar{x}$ be the image of $\tilde{x}$ in $K_0^{MW}(k)$.

We know that $K^{MW}_0(k) \isomt GW(k)$ \cite[Lemma 2.10]{MField}, and that $GW(k)$ can be realized as the fiber product $\Z \times_{\Z/2} W(k)$, where the homomorphism $\Z \to \Z/2$ is just reduction mod $2$, while the map $W(k) \to \Z/2$ is the mod $2$ rank homomorphism.  The kernel of $W(k) \to \Z/2$ is precisely the fundamental ideal $I(k)$.  By assumption $\bar{x}$ can be written $(1,\bar{x}')$ where $\bar{x}'$ is an element of $W(k)$ whose image in $\Z/2$ is $1$.

{\em Case 1.}  Suppose $k$ is not formally real, i.e., $-1$ is a sum of squares in $k$.  Under this hypothesis on $k$, the Witt ring $W(k)$ is local with unique maximal ideal equal to the kernel of the mod $2$ rank homomorphism \cite[Proposition 31.4(2)]{EKM}.  In particular, $\bar{x}'$ is invertible.  Since the oriented degree map is $GW(k)$-linear, after multiplication by $\bar{x}'^{-1}$, we can assume the image of the lift $\tilde{x}$ of $x$ is $1$.

{\em Case 2.}  Suppose $k$ is formally real.  Consider the diagram
\[
\xymatrix{
H^n_{\Nis}(X,\K^{MW}_n(\omega_X)) \ar[d]^{\tdeg}\ar[r]& H^n_{\Nis}(X,\K^M_n) \ar[d]^{\deg}\\
H^0_{\Nis}(\Spec k,\K^{MW}_0) \ar[r] & H^0_{\Nis}(X,\K^M_0).
}
\]
The Gersten resolution for $\K^{MW}_n(\omega_X)$ on $X$ gives rise to a surjection
\[
\bigoplus_{x \in X^{(n)}} GW(\kappa_x,\Lambda^n {\mathfrak m}_x/{\mathfrak m}_x^2) \longrightarrow H^n_{\Nis}(X,\K^{MW}_n(\omega_X)).
\]
The morphism $\tdeg$ can be lifted to a map
\[
\bigoplus_{x \in X^{(n)}} GW(\kappa_x,\Lambda^n {\mathfrak m}_x/{\mathfrak m}_x^2) \longrightarrow GW(k)
\]
defined as the sum of maps $GW(\kappa_x,\Lambda^n {\mathfrak m}_x/{\mathfrak m}_x^2) \to GW(k)$; see \cite[Corollaire 6.4.3]{Fasel1} for the construction of this homomorphism.  Fixing a trivialization of $\Lambda^n {\mathfrak m}_x/{\mathfrak m}_x^2$, the homomorphism just mentioned can be identified with the transfer homomorphism $GW(\kappa_x) \to GW(k)$ induced by viewing a symmetric bilinear form over $\kappa_x$ as a linear map over $k$ and composing with the field trace.

The map $H^n_{\Nis}(X,\K^{MW}_n(\omega_X)) \to H^n_{\Nis}(X,\K^M_n)$ is induced by the rank homomorphism
\[
\bigoplus_{x \in X^{(n)}} GW(\kappa_x,\Lambda^n {\mathfrak m}_x/{\mathfrak m}_x^2) \longrightarrow \bigoplus_{x \in X^{(n)}} \Z.
\]
If $x \in X^{(n)}$, then the local contribution at $x$ to the degree map is the homomorphism $\Z \to \Z$ given by multiplication by $[\kappa_x:k]$.  A $0$-cycle of degree $1$ on $X$ comes from an element of the form $\sum_{x \in X^{(n)}} n_x x$ where $n_x$ is zero for all but finitely many points in $X^{(n)}$ and the degrees of the extensions $[\kappa_x:k]$ are coprime.  Fix such a representative of our $0$-cycle of degree $1$ and let $I$ be the finite set of points $X^{(n)}$ that appear in our representation.  By definition, the restricted map $\oplus_{x \in I} \Z \to \Z$ is still surjective.  Thus, to prove our result, it suffices to prove that the corresponding map $\bigoplus_{x \in I} GW(\kappa_x) \to GW(k)$ is surjective.

Using the decomposition of $GW(k)$ as a fiber product of $\Z$ and $W(k)$, since the $\kappa_x$ have coprime degrees, we can assume that some $\kappa_x$ has odd degree over $k$.  Now, if $k$ is formally real, every finite extension $L$ of $k$ is simple.  If $L/k$ is a simple extension of odd degree, the Scharlau transfer $s_*: W(L) \to W(k)$ is a split surjection \cite[Lemme 6.4.4]{Fasel1}.
The transfer morphisms discussed are constructed in the following way: Given a $k$-linear homomorphism $\phi: L\to k$, one can define a transfer $t_\phi$ by viewing a symmetric bilinear $L$-form as a linear map $V\otimes V \to L$ and composing this with $\phi$. It is observed in the proof \cite[20.7]{EKM} that for any linear form $\phi$, there is an element $a\in L^\times$ such that $t_\phi(\mathfrak{b}) = s_*(a\mathfrak{b})$ for all symmetric bilinear forms $\mathfrak{b}$. That is, we may find and automorphism $m_a: W(L) \to W(L)$ such that $s_* \circ m_a = t_\phi$; this implies that all possible transfers (in particular the one associated to the field trace) are split surjective.
\end{proof}

\begin{footnotesize}
\bibliographystyle{alpha}
\bibliography{orientedzerocycles}
\end{footnotesize}
\end{document}